\documentclass[twoside,11pt]{article}

\usepackage[abbrvbib, preprint]{jmlr2e}
\usepackage{amsmath}
\usepackage{dsfont}

\usepackage{xcolor}

\DeclareMathOperator*{\argmin}{Argmin}

\newcommand{\prt}[1]{\left( #1 \right)}
\newcommand{\brkt}[1]{\left[ #1 \right]}
\newcommand{\crochets}[1]{\left\lbrace #1 \right\rbrace}
\newcommand{\ind}[1]{\mathds{1}_{#1}}
\newcommand{\esp}[1]{\mathbb{E}\brkt{#1}}
\newcommand{\esps}[2]{\mathbb{E}_{#1}\brkt{#2}}
\newcommand{\abs}[1]{\left\vert #1 \right\vert}
\newcommand{\Pb}[1]{\mathcal{P}\prt{#1}}
\newcommand{\R}{\mathbb{R}}
\newcommand{\Rm}{\mathcal{R}}
\newcommand{\F}{\mathcal{F}}
\newcommand{\G}{\mathcal{G}}
\newcommand{\N}{\mathbb{N}}
\renewcommand{\P}{\mathbb{P}}
\newcommand{\Q}{\mathbb{Q}}
\renewcommand{\S}{\mathcal{G}}
\newcommand{\cref}[1]{\hyperlink{#1}{$\prt{\mbox{#1}}$}}
\newcommand{\clabel}[1]{\hypertarget{#1}{$\prt{\mbox{#1}}$}}
\newcommand{\eqnref}[1]{Eq.~\eqref{#1}}
\newcommand{\thref}[1]{Theorem~\ref{#1}}
\newcommand{\sctref}[1]{Section~\ref{#1}}
\newcommand{\ssref}[1]{Subsection~\ref{#1}}
\newcommand{\Psh}{\Pb{\S,h}}
\renewcommand{\Rsh}{\Rm\prt{{\S,h}}}
\newcommand{\risk}{R}
\newcommand{\emp}{r}
\newcommand{\empsar}{r_{SAR}}
\newcommand{\Emp}{\widehat{R}_{n}}
\newcommand{\Empscar}{\widehat{R}_{n}^{\small SCAR}}
\newcommand{\Empscarp}{\widehat{R^\prime}_{n}^{\small SCAR}}
\newcommand{\Empsar}{\widehat{R}_{n}^{\small SAR}}
\newcommand{\Empc}{\overline{R_n}}
\newcommand{\Empsarc}{\overline{R_n^{\small SAR}}}
\newcommand{\Class}{Y}
\newcommand{\Lab}{S}
\newcommand{\hp}{h^\prime}
\newcommand{\Rshp}{\Rm\prt{{\S,\hp}}}
\newcommand{\cf}{cf.}

\newcommand{\ie}{\textit{i.e.}}

\jmlrheading{1}{2000}{1-48}{4/00}{10/00}{meila00a}{}

\ShortHeadings{Risk bounds for PU learning under Selected At Random assumption}{}
\firstpageno{1}

\begin{document}

\title{Risk bounds for PU learning under Selected At Random assumption}

\author{\name Olivier COUDRAY \email olivier.coudray@universite-paris-saclay.fr \\
       \addr Stellantis, Centre d'Expertise Métier et Région,
       Poissy, 78300, France\\
       \addr Université Paris-Saclay, CNRS, Inria,
       Laboratoire de mathématiques d’Orsay,
       Orsay, 91405, France
       \AND
       \name Christine KERIBIN \email christine.keribin@universite-paris-saclay.fr \\
       \addr Université Paris-Saclay, CNRS, Inria,
       Laboratoire de mathématiques d’Orsay,
       Orsay, 91405, France
       \AND
       \name Pascal MASSART \email pascal.massart@fondation-hadamard.fr\\
       \addr Université Paris-Saclay, CNRS, Inria,
       Laboratoire de mathématiques d’Orsay,
       Orsay, 91405, France
       \AND
       \name Patrick PAMPHILE \email patrick.pamphile@universite-paris-saclay.fr \\
       \addr Université Paris-Saclay, CNRS, Inria,
       Laboratoire de mathématiques d’Orsay,
       Orsay, 91405, France
       }

\editor{}

\maketitle

\begin{abstract}%   <- trailing '%' for backward compatibility 
Positive-unlabeled learning (PU learning) is known as a special case of semi-supervised binary classification where only a fraction of positive examples are labeled. The challenge is then to find the correct classifier despite this lack of information. Recently, new methodologies have been introduced to address the case where the probability of being labeled may depend on the covariates. In this paper, we are interested in establishing risk bounds for PU learning under this general assumption. In addition, we quantify the impact of label noise on PU learning compared to standard classification setting. Finally, we provide a lower bound on minimax risk proving that the upper bound is almost optimal.
\end{abstract}


\section{Introduction}

%Binary classification is a supervised machine learning task in which, given training observations with known \emph{classes} (\emph{positive} or \emph{negative}), one seeks to predict the class of new data. It is commonly assumed that the classes provided in the training set are correct. However, in many realistic situations, the observed \emph{labels} can be noisy: a case in point is when the class assignment in the training set is done manually and is then subject to errors.\\
Classic binary classification is a supervised machine learning task in which, from training observations with given \emph{classes} (\emph{positive} or \emph{negative}), one seeks to predict the class of new data. However, in many realistic situations, the observed \emph{classes} can be noisy. A case in point is when the class assignment is subject to errors.
 %In this paper, we are interested in a special case of label noise, occurring when only a fraction of positive instances is labeled. Hence, the other instances are \emph{unlabeled} meaning that their class is unknown. 
In this paper, we are interested in a special case of label noise, occurring when  a fraction of positive instances is labeled and none of the negative instances are. The \emph{unlabeled} instances are either positive or negative: their class is unknown.

\par This can be seen as a semi-supervised classification setting because only a fraction of the observations is labeled. This semi-supervised classification task is called \emph{Positive-Unlabeled Learning} (\emph{PU learning}). 
 %\\The objective of PU learning is to find the correct classifier despite the presence of this label noise.\\
The objective of PU learning is to build classifiers that find the right class (positive or negative) of a new data point given a training dataset of positive and unlabeled observations.

% This work was motivated by an application in mechanical design against fatigue. Fatigue is a slow damaging process occurring on a mechanical part subjected to repetitive loading and resulting in crack initiation. In order to check the good design of a mechanical part against fatigue bench tests are carried out. However, tests are long and expensive, therefore engineers prefer to resort to numerical simulations to predict which zones are critic (may fail during testing). The idea is then to use past simulation data combined with test results to build a classifier able to predict critical zones on new numerical models. In this scenario, experimental tests only provide a fraction of critical zones that effectively broke during testing. However, the absence of crack initiation is not an evidence of safety : maybe other zones could have initiated a crack if the tests had been carried on. Hence all the other zones need to be considered unlabeled. -> ref SIA et moins long (et après le reste des applis)

PU learning is used in situations where it is difficult or costly to obtain or identify reliable negative instances.
%\par For example in \emph{gene disease detection} \citep{10.1093/bioinformatics/bts504, Nikdelfaz2018May, Yang2014May}, the goal is to identify genes related to various human diseases. In these applications, some disease-related genes are already known and treated as positive labeled instances, but the class of the remaining genes is unknown. 
%A enlever \textcolor{blue}{For example, in the detection of genetic diseases (disease gene identification?), the objective is to identify the genes linked a disease: some disease-related genes are already known and treated as positive labeled instances, but the class of the remaining genes is unknown \citep[\cf][]{ 10.1093/bioinformatics/bts504, Yang2014May, Nikdelfaz2018May}.}
%\par In \emph{spam review detection}, as it is usually difficult to judge the authenticity of a review, several authors \citep{FusilierDonato2015Jul, Li2014Dec, He2020Feb} consider a PU-learning-based approach, using a set of identified fake reviews as positive labeled instances and the remaining ones as unlabeled instances.
%\textcolor{red}{je trouve cet exemple peu caractéristique du PU !!! je préfèrerais la détection d'anomalie...}
For example, in the diagnosis of a disease, given the incubation period of the disease, a patient with a negative test may still be carrying the disease \citep[\cf][]{Chen2020Jun}.
PU learning approach could be useful in \emph{fatigue design of structures} in mechanics where testing can prove the presence of design flaws on a mechanical part, but cannot prove its absence. In the automotive industry, fatigue tests are performed to determine if a part is critical: if a crack is observed before the end of the test then the part is declared critical. But if no crack is observed it does not mean that the part is not critical. It may be possible to observe a crack by extending the test \citep[\cf][]{coudray:hal-03483277}.
Other applications of PU learning exist in \emph{spam review detection} \citep[\cf][]{Li2014Dec, FusilierDonato2015Jul,  He2020Feb}, \emph{text classification} \citep[\cf][]{Liu2002, liu2003building}, \emph{gene disease identification} \citep[\cf][]{ 10.1093/bioinformatics/bts504, Yang2014May, Nikdelfaz2018May} and \emph{anomaly detection} \citep[\cf][]{Ferretti2014Sep, Luo2018Oct, Jiang2018Nov}. 

PU learning is therefore much more difficult than learning from fully labeled data. The situation is asymmetric as one usually wants to understand the positive class in contrast with an unidentified negative class. The number of positive labeled examples is critic as it consists in the only reliable information.

Different methodologies have been developed to address PU learning. A first class of heuristic methods proceeds in two steps. The first step consists in identifying reliable negative instances among the unlabeled observations: various methods exists like \emph{Spy} \citep[\cf][]{Liu2002} or \emph{Rocchio} \citep[\cf][]{li_learning_2003} methods. In the second step, a standard supervised or semi-supervised classification method is used to  build the PU classifier from the positive labeled instances, reliable negative instances and the remaining unlabeled ones.  A typical choice is Support Vector Machine (SVM). Some methods repeat iteratively both steps until convergence. \citet{Bekker2020Apr} gave an exhaustive list of existing methods for both steps. These methods are supported by good empirical results, but theoretical guarantees are not discussed. 

Another class of methodologies resorts to a modeling of label noise and adapts existing supervised classification methods to PU learning setting. 
%à mettre plus loin (formalisation du pb de PU learning)
% \textcolor{blue}{Missing labels in PU learning can arise from different settings. In the \emph{two-sample setting}, the positive and unlabeled instances are sampled separately and are therefore not identically distributed : it is a \emph{case-control} situation. In the \emph{one-sample setting}, all the instances are i.i.d and some positive instances are labeled.}
Most existing PU learning methods in this category assume that the probability for a positive instance to be labeled is constant and thus independent from the covariates. This situation is called Selected Completely At Random (SCAR). However, in certain cases, the probability of a positive instance being labeled is influenced by its covariates. For example, in the diagnosis of a disease, a carrier of the disease with symptoms is more likely to see a doctor and be diagnosed than a carrier who is asymptomatic. This situation with a selection bias is called Selected At Random (SAR).
Under the SCAR assumption, since the noise for positive instances is constant, the probability for an instance to be labeled is then proportional to the probability for it to be positive: fully labeled classification and PU learning are then connected. Hence, some algorithms use this property to derive consistent classifiers:  \cite{blanchard2010semi} use Neyman-Pearson classification and \cite{Plessis2014Jan} rewrite PU learning as a cost-sensitive binary classification that can be solve through empirical risk minimization. These approaches are supported by theoretical guaranties: consistency and risk bounds. 
%In practice, this is rarely the case, hence the class prior is estimated from the data \citep[\cf][]{du2014class, jain_estimating_2016, ramaswamy_mixture_2016, bekker_estimating_2018}.
As mentionned above, SCAR assumption is unlikely to hold in many practical situations.
%: for instance, in automatic diagnosis, a patient with severe symptoms is more likely to see a doctor and then be diagnosed. Hence, the set of diagnosed patients (\emph{labeled positives}) is biased \citep{Bekker2018Sep}. \textcolor{blue}{The SCAR assumption implies that labeled instances are uniformly selected from positive instances. Unfortunately in many situations there is a selection bias: a carrier of the disease with symptoms is more likely to be detected than a carrier but asymptomatic individual.The SAR hypothesis can then be used } 
Recently, several publications have addressed PU learning when the probability of being labeled is instance-dependent \citep{Bekker2018Sep, gong_instance-dependent_2021}.
%In particular, \cite{Bekker2018Sep} and \cite{gong_instance-dependent_2021} studied theoretical properties of PU learning under selection bias: consistence and risk bounds.  \textcolor{red}{j'aurais mis cela après l'approche theorique avec SCAR}

From a theoretical point of view, risk bounds in standard classification setting have been extensively studied in the literature. The convergence rate of the excess risk in classification is known to be less than a quantity proportional to $\sqrt{1/n}$ where $n$ is the size of the training set \citep[\cf][]{lugosi2002}. In addition, this rate can be refined, reaching $1/n$ in the noiseless situation \citep{Massart2006Oct}. Finally, these rates are proved to be optimal in the minimax sense \citep[\cf][]{lugosi2002, Massart2006Oct}. 

%\par In the past few years, several papers have studied excess risk upper bounds for PU learning classifiers under SCAR assumption in the case-control setting, i.e. when the training set consists of $n_L$ i.i.d. positive labeled instances and $n_U$ i.i.d. unlabeled instances \citep{Plessis2014Jan, gong_instance-dependent_2021}. The observations are then no longer identically distributed. 
Missing labels in PU learning can arise from different settings. In the \emph{two-sample setting}, the positive and unlabeled instances are sampled separately and are therefore not identically distributed: it is a \emph{case-control} situation. In the \emph{one-sample setting}, all the instances are i.i.d and some positive instances are labeled.
In the past few years, several papers have studied excess risk upper bounds for PU learning classifiers in the case-control setting. In this case, \cite{Plessis2014Jan} showed a convergence rate in $\mathcal{O}\prt{\sqrt{1/n_L} + \sqrt{1/n_U}}$. More recently, \cite{Bekker2018Sep} and \cite{gong_instance-dependent_2021} studied theoretical properties of PU learning under selection bias with specific assumptions: the former establishes an upper bound on an empirical risk minimizer under partial knowledge of the labeling mechanism, the latter focus on a parametric model.

In this work, we focus on PU learning in the one-sample setting.
We provide a theoretical study of PU learning under selection bias, meaning that the label noise due to PU learning is instance-dependent (Selected At Random assumption). Unlike \cite{gong_instance-dependent_2021}, we do not make parametric assumptions. 
%Assuming partial knowledge on the labelling mechanism, a consistent classifier can be built by minimizing an unbiased empirical risk. 
Contrary to \cite{Bekker2018Sep} who focused on the deviations between PU learning empirical risk and fully supervised empirical risk, we provide an upper bound on the excess risk. The novelty of this result also lies in its ability to explicitly quantify the impact of label noise and to show that fast convergence rates can be achieved under noise conditions similar to \cite{Massart2006Oct}. Finally, we discuss the optimality of this result by identifying a lower bound on the minimax risk.

The paper is organized as follows. In \sctref{context}, we define the standard binary classification setting and recall some existing risk bounds. In \sctref{unbiased}, we move to PU learning setting, discuss the bias issue with labeled-unlabeled classification and introduce an unbiased empirical risk. In \sctref{results}, we present the main results of this paper: a general upper bound on the excess risk for PU learning under instance-dependent label noise, and a lower bound on minimax risk. In \sctref{conclusion}, we conclude and discuss some future perspectives.

\section{Standard classification setting}\label{context}

In this section, we introduce the standard classification setting and recall risk bounds results. This will be the opportunity to introduce general notations used throughout the paper.

\subsection{General setting}

Let $(X_1, \Class_1),...,(X_n, \Class_n)$ be independent couples of random variables in $\R^d\times \{0,1\}$ identically distributed according to some unknown distribution denoted $\P$. For each $i$, $X_i$ is a \emph{covariate} vector with marginal distribution $\P_X$ and $\Class_i$ is the \emph{class}, either \emph{negative} ($\Class_i=0$) or \emph{positive} ($\Class_i=1$). Let $\alpha=\P\prt{\Class=1}$ denote the class prior. Using $\P_0$ ($\P_1$) the conditional distribution of $X$ given that the class is negative, $\Class=0$ (positive, $\Class=1$), we write the convenient decomposition:
\begin{equation}\label{decomp}
    \P_X = (1-\alpha)\P_0 + \alpha \P_1\ .
\end{equation}

In classification, the goal is to find a classifier, \textit{i.e.} a binary function $g:\R^d \rightarrow \crochets{0,1}$, minimizing some risk function $\risk$. In this paper, $\risk$ will denote the misclassification risk:
\begin{equation}\label{risk}
    \risk\prt{g} = \P\prt{g(X)\neq \Class}\ .
\end{equation}
Given the regression function $ \eta(x) = \P\prt{\Class=1 \vert X=x} $, the minimizer of misclassification risk is Bayes classifier $g^*$ that depends explicitely on $\P$:
$$ g^*(x) = \ind{\eta(x)\geq \frac{1}{2}}. $$

In order to assess how close a given classifier $g$ is to the optimal one $g^*$, we are interested in the excess risk $\ell\prt{g, g^*}$:
$$ \ell\prt{s, g^*} = \risk\prt{g} - \risk\prt{g^*}. $$

Since $\P$ is unknown, neither $g^*$ nor the risk function $\risk$ can be computed. We rely instead on the training sample $(X_1, \Class_1),...,(X_n, \Class_n)$ to build a classifier $\widehat{g}$. Let $\emp\prt{g, \prt{X,\Class}} = \ind{g(X)\neq \Class}$ the misclassification error for one observation, the true risk $\risk$ can be estimated by the empirical mean:
\begin{equation*}
    \Emp\prt{g} = \frac{1}{n}\sum_{i=1}^n \emp\prt{g, \prt{X_i,\Class_i}}\ .
\end{equation*}

An empirical classifier $\widehat{g}$ is then identified as a minimizer of the empirical risk over a predefined class of classifiers $\S$.
$$ \widehat{g} \in \argmin_{g \in \S}\ 
\Emp\prt{g}. $$
This procedure is known as Empirical Risk Minimization. Let $g^{\S}$ be the minimizer of the true risk $\risk$ over $\S$. The excess risk of the classifier $\widehat{g}$ can be decomposed as follows:
\begin{equation}
    \ell(\widehat{g}, g^*) = \prt{\risk\prt{g^{\S}} - \risk\prt{g^*}} + \prt{\risk\prt{\widehat{g}} - \risk\prt{g^{\S}}}
\end{equation}
where the first term is the approximation error depending on $\S$ and the second one is the statistical error. Since we are only interested in assessing the statistical error, we assume that Bayes classifier $g^*$ belongs to $\S$, hence the first term vanishes. Note that $\ell\prt{\widehat{g}, g^*}$ depends on $\P$ (through the risk $\risk$) and on the training sample $(X_1, \Class_1),...,(X_n, \Class_n)$.

\subsection{Risk bounds in standard classification}\label{riskstd}
In order to assess the convergence rate of the excess risk $\ell\prt{\widehat{g}, g^*}$ in a non-asymptotic framework, we need an upper bound on $\esp{\ell\prt{\widehat{g}, g^*}}$. Note that the expectation is taken with respect to the distribution of the training sample $\P^{\otimes n}$. And the upper bound needs to be uniform over a set of distributions $\P$.
%As the upper bound needs to be uniform over a set of probability distributions, 
We introduce $\Pb{\S}$ a set of probability distributions on $\R^d \times \{0,1\}$ such that $g^*$ belongs to $\S$.
In this case, \cite{lugosi2002} proved that for some absolute constant $C_1 > 0$:
\begin{equation}\label{general-upper}
 \sup_{\P \in \Pb{\S}} \esp{\ell\prt{\widehat{g}, g^*}} \leq C_1 \sqrt{\frac{V}{n}}\ ,
\end{equation}
where $V$ is the \emph{Vapnik-Chervonenkis dimension} of $\S$ (VC dimension, see \citealp[Chapter~3]{vapnik1999}).
We recall that the VC dimension is the maximum integer $V$ such that there exists $V$ points $x_1,...,x_V$ in $\R^d$ \emph{shattered} by $\S$, namely classified in every way possible by elements of $\S$. In other words: $$V = \sup_{v \in \N^*}\crochets{v\ s.t.\ \exists\, x_1,...,x_v \in \R^d,\ \abs{\crochets{\prt{g(x_1),...,g(x_v)},\ g \in \S}} = 2^v}\ .$$
Note that $V$ measures the complexity of class $\S$ and has to be finite for \eqnref{general-upper} to be meaningful, which we assume for the rest of the paper.

The upper bound in \eqnref{general-upper} remains true regardless of the form of the regression function $\eta$. Actually, $\eta$ is closely linked to the \emph{label noise}: when $\eta(x)$ is close to $1/2$, the observed class can be positive or negative with probability close to $1/2$, which makes the classification of $x$ more difficult.
%if $\eta(x)\geq \frac{1}{2}$ ($\eta(x) < \frac{1}{2}$), then with probability $1-\eta(x)$ ($\eta(x)$) the observed class will differ from the one given by Bayes classifier $g^*(x)$. 
Hence, the closer $\eta$ is to $1/2$, the noisier the observed class is. Moreover \cite{Massart2006Oct} showed that whenever $\eta(x)$ is uniformly and symmetrically bounded away from $1/2$ by a quantity $h>\sqrt{V/n}$, the upper bound on the risk excess can be improved. 
Let $\Psh$ denote the subset of probability distributions in $\Pb{\S}$ such that for every $x \in \R^d$, $\abs{2 \eta(x) - 1}\geq h$. \cite{Massart2006Oct} showed that there exists an absolute constant $C_2 > 0$ such that:
\begin{equation}\label{massart-upper}
 \sup_{\P \in \Psh} \esp{\ell\prt{\widehat{g}, g^*}} \leq C_2 \frac{V}{n h}\prt{1+\log\prt{\frac{nh^2}{V}}}\ .
\end{equation}
Hence, as $h$ gets higher, the label noise gets smaller, and the convergence rate can be improved up to $1/n$ letting aside the logarithm. However, when $h$ is smaller than $\sqrt{1/n}$, \eqnref{general-upper} remains better. \eqnref{massart-upper} provides a fine control on the excess risk depending on the difficulty of the classification task, accounted through $h$.

A lower bound was obtained by \cite{lugosi2002}, extended by \cite{Massart2006Oct}, allowing to prove the optimality of the convergence rates. Note that optimality of the refined bound \eqnref{massart-upper} is up to the logarithmic term.

\section{PU learning context}\label{unbiased}

In standard classification setting, the classes $(\Class_i)_{1\leq i\leq n}$ are observed. This is no longer the case in PU learning where only an incomplete set of positive data is available, the remaining is unlabeled. For each $i$, the observed label $\Lab_i$ is $1$ if the class $\Class_i$ is positive and \emph{selected} (\textit{i.e.} labeled). Otherwise, the label $\Lab_i$ is $0$. The true classes are affected by a class-dependent (thus asymmetric) label noise. The probability for a positive instance to be labeled is generally called the \emph{propensity} \citep{Bekker2020Apr} and it may depend on the covariates:
\begin{equation} 
e(x) = \P\prt{\Lab=1\vert \Class=1, X=x}.
\end{equation}
Negative instances are never labeled: 
\begin{equation*}
\P\prt{\Lab=1\vert \Class=0, X=x} = 0\ .
\end{equation*}
Note that the regression function associated with $S$, $\tilde{\eta}(x) = \P\prt{S=1\vert X=x}$ depends on this additional label noise:
\begin{equation}\label{eta2}
    \tilde{\eta}(x) = e(x)\,\eta(x)\ .
\end{equation}
The objective of PU learning is to use the incomplete information $(X_1, S_1),...(X_n,S_n)$ to build a classifier able to predict the class $\Class$ given a new instance with covariates $X$.

This concept of completely asymmetric label noise was first pointed out by \cite{elkan_learning_2008}. It is now common to define two general types of assumptions: Selected Completely At Random (SCAR) and Selected At Random (SAR).
\begin{enumerate}
\item[SCAR:] PU learning without selection bias. The propensity $e(x)=e$ does not depend on the covariates $x$. This applies in situations where any instance has an equal probability to be selected (labeled). In this case, the conditional distributions of $X$ given $\Class=1$ ($\P_1$) and given $S=1$ ($\tilde{\P}_1$) are the same. In other words, labeled instances are a representative sub-sample of positive instances.
\item[SAR:] PU learning with selection bias. The probability for an instance to be selected depends on its covariates. Hence, labeled instances are a biased sample of positive instances. For example, in mechanical design, a specimen subjected to a higher stress is more likely to break, which results in a higher probability of a crack being detected. This is clearly a situation where SCAR assumption does not hold.
\end{enumerate}

In this section, we focus on the definition of loss functions that enable learning in PU learning setting. After explaining why labeled-unlabeled classifiers are limited, we will introduce an unbiased empirical risk for PU learning under SCAR assumption \citep[\cf][]{Plessis2014Jan}, which generalizes to SAR assumption \citep[\cf][]{Bekker2018Sep}. 

\subsection{Bias issue with labeled-unlabeled classification}

A natural idea to address a PU learning problem is to consider labeled instances as positive and every unlabeled instances as negative. Standard classification methods then allow to identify a classifier $\widehat{g}_{NT}$. In the literature, such a classifier is called a \emph{nontraditional classifier} \citep{elkan_learning_2008} because it is meant to give good predictions on $S$ instead of $\Class$. As the number of training examples increases, we can then expect $\widehat{g}_{NT}$ to get closer to Bayes classifier $\tilde{g}^*$ for the classification of $S$ given $X$ which is not what we are looking for. Indeed, $\tilde{g}^*$ is \textit{a priori} different from $g^*$ as the regression function $\tilde{\eta}(x)=\prt{Y=1 \vert X=x}$ is different from $\eta(x)$ (\cf\  \eqnref{eta2}).

Nevertheless, in specific situations, the nontraditional classifier is robust to PU learning label noise. \cite{Cannings} showed for example that $\tilde{g}^* = g^*$ if:
\begin{equation}\label{Canningscond}
e(x) \geq \frac{1}{2 \eta(x)}, \mbox{ for all } x \in \R^d\mbox{ such that }\eta(x)\geq \frac{1}{2}\ .
\end{equation}
Note that this is part of a more general result from \cite{Cannings} that encompasses binary classification with asymmetric and instance-dependent label noise.
Under the condition from \eqnref{Canningscond}, any consistent non-traditional classifier is a consistent traditional classifier. In other words, as the training sample size increases, $\widehat{g}_{NT}$ gets closer to $\tilde{g}^*$ which is identical to $g^*$.\\

This condition requires every positive instance ($\eta(x)>\frac{1}{2}$) difficult to classify ($\eta(x)$ close to $\frac{1}{2}$) to have propensity close enough to $1$. Instances easier to classify ($\eta(x)$ close to $1$) can undergo label noise without harming the consistence. However, the label noise cannot exceed $\frac{1}{2}$ or, in other words, the propensity can never be smaller than $\frac{1}{2}$. 

This condition is thus restrictive in the context of PU learning under SAR assumption for two main reasons. On the one hand, in many realistic situations, the propensity (\textit{i.e.} the probability for a positive instance to be labeled) is correlated to the difficulty of classifying the observation. A positive instance difficult to classify tends to have low propensity which clearly violates the condition given in \eqnref{Canningscond}. On the other hand, we cannot expect the propensity to be greater than $\frac{1}{2}$. In \emph{text classification} or \emph{spam review detection}, as the process of labelling is both difficult and time consuming, only a small fraction of positive instances gets labeled, which suggests a propensity lower than $\frac{1}{2}$.

Before dealing with convergence rates, it is crucial to have methods for building consistent classifiers under more general conditions than \eqnref{Canningscond}.

\subsection{Unbiased empirical risk minimization under SCAR assumption}

In this subsection, we assume that SCAR assumption is satisfied, which means that the propensity is constant:
\begin{equation*}
    e(x)=e_m > 0\ .
\end{equation*}
In order to compensate for label noise due to PU Learning under SCAR assumption, \cite{Plessis2014Jan} showed in the case-control setting that a consistent classifier can be found by minimizing an unbiased version of the risk. Using the convenient decomposition of $\P_X$ distribution (\eqnref{decomp}), the misclassification risk can be rewritten only with $\P_X$ and $\P_1$.

\begin{align}\label{decomp2}
    \risk\prt{g} &= \alpha\, \P_1\prt{g(X)\neq 1} + (1-\alpha)\, \P_0\prt{g(X)\neq 0} \nonumber\\
    %&= \alpha\, \P_1\prt{g(X)\neq 1} + \P\prt{g(X)\neq 0} - \alpha\, \P_1\prt{g(X)\neq 0} \\
    &= \alpha \prt{\P_1\prt{g(X)\neq 1} - \P_1\prt{g(X)\neq 0}} + \P_X\prt{g(X)\neq 0}.
\end{align}

Therefore, as labeled instances are a representative sub-sample of positive instances, a consistent classifier can be found by minimizing the following risk:
\begin{equation*}
    \Empscar(g) = \frac{\alpha}{N_L}\sum_{i=1}^n \ind{\Lab_i=1}\brkt{\ind{g(X_i) \neq 1} - \ind{g(X_i) \neq 0}} + \frac{1}{n}\sum_{i=1}^n \ind{g(X_i) \neq 0}
\end{equation*}
where $N_L = \sum_{i=1}^n \ind{\Lab_i=1}$ is the number of labeled instances. Note that \cite{Plessis2014Jan} considered the case-control setting where the number of labeled instances $N_L$ is fixed which is slightly different from our setting. 
One of the main properties of $\Empscar(g)$ is that it is an unbiased estimate of the true risk, as we have:
\begin{equation*}
    \esp{\Empscar(g)} = \P\prt{g(X)\neq \Class}\ .
\end{equation*}
The proof of \cite{Plessis2014Jan} extends to the one sample setting where $N_L$ is random:

\begin{subequations}
\begin{align}
    \esp{\Empscar(g)} &= \sum_{i=1}^n\esp{\frac{\alpha}{N_L}\ind{\Lab_i=1}\esp{ \ind{g(X_i)\neq 1}-\ind{g(X_i) \neq 0}\, \vert\, \Lab_i}} + \P_X\prt{g(X)\neq 0} \notag\\
    &= \sum_{i=1}^n\esp{\frac{\alpha}{N_L}\ind{\Lab_i=1}\brkt{\P\prt{g(X_i)\neq 1\, \vert\, S_i}-\P\prt{S_i=1,g(X_i)\neq 0\, \vert\, S_i}}} \notag\\
    &+ \P_X\prt{g(X)\neq 0} \notag\\
    &= \alpha \sum_{i=1}^n \esp{\frac{\ind{\Lab_i=1}}{N_L} \prt{\P_1\prt{g(X)\neq 1} - \P_1\prt{g(X)\neq 0}}}
    + \P_X\prt{g(X)\neq 0} \label{eq-scar1} \\
    &= \alpha \brkt{\P_1\prt{g(X)\neq 1} - \P_1\prt{g(X)\neq 0}} + \P_X\prt{g(X)\neq 0} \label{eq-scar2} \ .
\end{align}
\end{subequations}
\eqnref{eq-scar1} results from the fact that under SCAR assumption the conditional distribution of $X$ given $\Lab=1$ is the same as the conditional distribution of $X$ given $\Class=1$ ($\P_1$).  Finally, \eqnref{eq-scar2} matches the decomposition of \eqnref{decomp2}, ending the proof.

Computing the risk $\Empscar$ requires $\alpha$ to be known. Alternatively, another empirical risk can be written:
\begin{equation}\label{SCARrisk2}
    \Empscarp(g) = \frac{1}{n}\sum_{i=1}^n\brkt{ \frac{\ind{\Lab_i=1}}{e_m}\prt{\ind{g(X_i) \neq 1} - \ind{g(X_i) \neq 0}} + \ind{g(X_i) \neq 0}}\ .    
\end{equation}
This risk remains unbiased and consistent but requires the knowledge of the constant propensity $e_m$ instead of the class prior $\alpha$. The unbiasedness of $\Empscarp$ will be proved in \ssref{sarloss} as a special case of the more general SAR setting.

\subsection{Extension to PU learning under SAR assumption}\label{sarloss}

For now, PU learning under SAR assumption is a difficult problem and there are only few results in the literature \citep[\cf][]{Bekker2018Sep, He2018Aug, gong_instance-dependent_2021}. We recall that empirical risk minimization under SCAR assumption requires extra knowledge on the model (class prior or propensity).  In practice, these parameters are usually estimated \citep[\cf][]{blanchard2010semi, du2014class, jain_estimating_2016, ramaswamy_mixture_2016, bekker_estimating_2018}. In order to provide a consistent empirical risk in SAR setting, additional assumptions are needed to avoid identifiability issues. In the literature, different settings have been studied. \cite{He2018Aug} assume that the propensity $e(x)$ is an increasing function of $\eta(x)$. \cite{Bekker2018Sep} and \cite{gong_instance-dependent_2021} suggest a \emph{parametric} model on the propensity. \cite{Bekker2018Sep} also study the case where the propensity is known for labeled instances which enables an empirical risk minimization approach similar to \cite{Plessis2014Jan}.

In this paper, following \cite{Bekker2018Sep}, we will focus on PU learning under SAR assumption where the propensity is known for labeled instances. We argue that this setting is sufficient to derive interesting risk bounds and assess the difficulty of PU learning tasks. However restrictive this assumption may seem, we insist that only the propensity for labeled instances is needed, therefore an exhaustive knowledge of propensity is not required. In practice, the propensity can be estimated using prior knowledge on the labeling mechanism (when available) or by defining a parametric model on the propensity \citep{Bekker2018Sep, gong_instance-dependent_2021}.
%Considering our application in reliability, we have at our disposal different repetitions of tests carried out with different load intensity (different severities). Besides, the probability distribution of a critical zone's lifetime is already well modeled and involves a few parameters that were already identified using tests on simpler mechanical parts. We then have useful information to correctly estimate the propensity for each labeled instance. Basically, we have extra knowledge on the labelling mechanism.

Under this assumption, \cite{Bekker2018Sep} generalized the empirical risk in \eqnref{SCARrisk2} to obtain an unbiased empirical risk for PU learning under SAR assumption. More particularly, they define the following loss function:
\begin{equation*}
\begin{split}
    \empsar\prt{g,\prt{X,\Lab}} &= \frac{\ind{\Lab=1}}{e(X)}\prt{\ind{g(X)\neq 1} - \ind{g(X)\neq 0}} + \ind{g(X) \neq 0}\\
    &= \frac{\ind{\Lab=1}}{e(X)}\prt{2\, \ind{g(X)\neq 1} - 1} + \ind{g(X) \neq 0}\ .
\end{split}
\end{equation*}
The empirical risk is then the empirical mean:
\begin{equation}\label{risksar}
    \Empsar(s) = \frac{1}{n} \sum_{i=1}^n \empsar\prt{g,\prt{X_i,\Lab_i}}\ .
\end{equation}
This time, the labeled instances are weighted by the inverse of there propensity. Clearly, $\Empscarp$ in \eqnref{SCARrisk2} is a special case of $\Empsar$ under SCAR assumption ($e(x)=e_m$).

\cite{Bekker2018Sep} studied maximum deviations between this latter empirical risk $\Empsar$ and the empirical risk for standard classification $\Emp$. 
%In particular, using that the conditional expectation of $\Empsar(g)$ given $(X_i, \Class_i)_{1\leq i\leq n}$ is $\Emp$. 
They then used concentration inequalities to derive an upper bound with high probability on the deviations between the two quantities. As we are interested in studying directly the deviations between $\Empsar$ and the true risk $\risk$, we compute the total expectation of $\esp{\Empsar(g)} = \esp{\empsar\prt{g,\prt{X,\Lab}}}$ shedding light on the fact that for any $g$, $\Empsar(g)$ is an unbiased estimate of the true risk $\risk(g)$.
\begin{align*}
\allowdisplaybreaks
    \esp{\empsar\prt{g,\prt{X,\Lab}}} &= \esp{\esp{\empsar\prt{g,\prt{X,\Lab}} \vert X}} \\
    &= \esp{\frac{1}{e(X)}\prt{\ind{g(X)\neq 1} - \ind{g(X)\neq 0}}\P\prt{\Lab=1 \vert X}} + \P_X\prt{g(X) \neq 0} \\
    &= \esp{\frac{1}{e(X)}\prt{\ind{g(X)\neq 1} - \ind{g(X)\neq 0}}\eta(X) e(X)} + \P_X\prt{g(X) \neq 0}\\
    &= \esp{\prt{\ind{g(X)\neq 1} - \ind{g(X)\neq 0}} \ind{\Class=1}} + \P_X\prt{g(X) \neq 0}\\
    &= \alpha \prt{\P_1\prt{g(X)\neq 1} - \P_1\prt{g(X)\neq 0}} + \P_X\prt{g(X) \neq 0}\\
    &=R(g) \ .
\end{align*}

where the last line comes from \eqnref{decomp2}. Then, $\Empsar$ is indeed unbiased:
\begin{equation}\label{sar-unbiased}
    \esp{\Empsar(g)} = \frac{1}{n} \sum_{i=1}^n \esp{\empsar\prt{g,\prt{X_i,\Lab_i}}}=\P\prt{g(X)\neq \Class} = \risk(g)\ .
\end{equation}

\section{Main results}\label{results}
We are now in position to state our results. We first present an upper bound on the excess risk for PU learning under SAR assumption. We then show that the rate achieved is almost optimal by providing a lower bound on the minimax risk. Both bounds quantify explicitly the impact of label noise due to PU learning.

\subsection{An upper bound for the PU learning excess risk under SAR assumption}

We recall that, in PU learning, the true classes $(\Class_i)_{1\leq i\leq n}$ are no longer available for training. A classifier is then built as a minimizer of the unbiased empirical risk introduced in \eqnref{risksar}:
\begin{equation}\label{minsar}
    \widehat{g}_{PU} \in \argmin_{g \in \S} \Empsar(g)\ .
\end{equation}

We recall that the risk $\Empsar$ is unbiased (\eqnref{sar-unbiased}) and we will denote $\Empsarc$ the centered empirical risk:
\begin{equation*}
    \Empsarc(g)=\Empsar(g)-\P\prt{g(X) \neq \Class}\ .
\end{equation*}

\cite{Bekker2018Sep} study the deviations between $\Empsar\prt{\widehat{g}_{PU}}$ and $\Emp\prt{\widehat{g}_{PU}}$, and provide an upper bound in the case where $\S$ is a \emph{finite} family of classifiers. Besides, the influence of $e(\cdot)$ on the upper bound is not discussed. Our objective here is to provide a uniform upper bound on $\ell(\widehat{g}_{PU}, g^*)$ and show explicitly its dependence in $e(\cdot)$. In our setting, $\S$ is an \emph{infinite} set of functions. Its complexity is controlled by its VC dimension $V<+\infty$. Following \cite{Massart2006Oct}, we consider the following separability assumption which is key to work with the possibly uncountable class $\S$:
\begin{enumerate}
\item[\clabel{$A_1$}] \quad There exists a countable subset $\S^\prime$ dense in $\S$ in the sense that for each $g \in \S$, there exists a sequence $\prt{g_k}_{k \geq 0}$ such that, for every $(x,s) \in \R^d\times \crochets{0,1}$:
\begin{equation*}
    \empsar\prt{g_k,\prt{x,s}} \underset{k \rightarrow +\infty}{\longrightarrow} \empsar\prt{g,\prt{x,s}}\ .
\end{equation*}
\end{enumerate}
In addition, we want our upper bound on the excess risk to explicitly account for the difficulty of the classification task. Then, as $\abs{2\eta(x)-1}$ quantify the difficulty of classifying $x$, we introduce the following assumption \citep{Massart2006Oct}:
\begin{enumerate}
\item[\clabel{$A_2$}] \quad $\exists\, h>0,\ \forall x \in \R^d,\ \abs{2\eta(x)-1} \geq h\ .$
\end{enumerate}
Assumption \cref{$A_2$} will be referred as \emph{Massart noise} assumption in the rest of the paper.\\

We are now able to state our upper bound for PU learning under SAR assumption.
\begin{theorem}[Upper risk bound for PU learning under SAR assumption]\label{th1}
~\\
Let $\widehat{g}_{PU}$ be a minimizer of the unbiased empirical risk for PU learning under SAR assumption:
\begin{equation*}
    \widehat{g}_{PU} \in \argmin_{g \in \S} \Empsar(g)\ .
\end{equation*}
Suppose that separability \cref{$A_1$} and Massart noise \cref{$A_2$} assumptions hold, and that the propensity $e(\cdot)$ is greater than $e_m>0$. Then, we have the following upper bound on the excess risk:
\begin{equation}\label{upper}
    \esp{\ell\prt{\widehat{g}_{PU}, g^*}} \leq \kappa_1 \brkt{\frac{V}{n\, e_m\, h} \prt{1+\log\prt{\frac{n\,h^2}{V} \vee 1}} \wedge \sqrt{\frac{V}{n\, e_m}}}
\end{equation}
where $\kappa_1>0$ is an absolute constant.
\end{theorem}

\paragraph*{Remarks:}
The upper bound in \eqnref{upper} is uniform on the set of probability distributions for which $g^* \in \S$ and Massart noise condition \cref{$A_2$} is satisfied with constant $h$ ($\Psh$). This can be re-written as follows:
\begin{equation}\label{upper-unif}
    \sup_{\P \in \Psh} \esp{\ell\prt{\widehat{g}_{PU}, g^*}} \leq \kappa_1 \brkt{\frac{V}{n\, e_m\, h} \prt{1+\log\prt{\frac{n\,h^2}{V} \vee 1}} \wedge \sqrt{\frac{V}{n\, e_m}}}
\end{equation}

Note that the assumption $e(x)\geq e_m$ is an additional assumption on the label noise. As the biased regression function is $\tilde{\eta}(x)=\eta(x)\, e(x)$ (\cf\  \eqnref{eta2}), this assumption together with assumption \cref{$A_2$} control the difficulty of the PU learning task.

In \eqnref{upper}, the convergence rate is of order $\mathcal{O}\prt{\frac{V}{n\,h\, e_m}}$ (if we let aside the logarithmic term) when $h$ is higher than $\sqrt{V/n\,e_m}$. When $h$ becomes smaller than $\sqrt{V/n\,e_m}$, the rate is of order $\mathcal{O}\prt{\sqrt{V/n\,e_m}}$. These two regimes are analogous to standard classification risk bounds as recalled in \ssref{riskstd}. In particular, when $e_m = 1$, all positive examples are labeled and we are then in a standard classification setting ($\Lab=\Class$). In this case, the upper bound exactly matches the known upper bound rates in standard classification setting (\eqnref{massart-upper} and \eqnref{general-upper}). Conversely, as $e_m$ gets lower, the upper bound increases. This means without surprise that PU learning deteriorates the generalization bound: \thref{th1} quantifies this effect through the coefficients $1/e_m$ and $1/\sqrt{e_m}$.\\

Let $N_L$ be the number of labeled instances in the training set. Under SCAR assumption ($e(x)=e_m$), $n\, e_m$ from \eqnref{upper} is linked to the expectation of the number of labeled instances in the training set:
\begin{equation*}
    \esp{N_L} = \esp{\sum_{i=1}^n \ind{\Lab_i=1}} = n\, \P\prt{\Lab=1} = n\,\alpha\,e_m
\end{equation*}
where $\alpha=\P\prt{\Class=1}$ is the class prior. This illustrates a natural intuition on PU learning: the upper bound on the excess risk is related to the number of fully labeled examples. Hence, good prediction performances cannot be expected if the number of labeled examples among the positives is too low, or equivalently if the propensity is too low.

The detailed proof of \thref{th1} can be found in Appendix \ref{proofth1}. It consists in establishing controls on the variance of increments of $\empsar(\cdot)$ and uniform bounds on the empirical process $\prt{\overline{\Gamma}(g)}_{g \in \S}$. A general risk bound result for empirical risk minimizers is then applied.\\
%the application of a general upper bound theorem for empirical risk minimizers from \cite{Massart2006Oct}. To do so, we verify the applicability of this result in PU learning, which requires uniform bounds on the empirical process $\prt{\overline{\Gamma}(g)}_{g \in \S}$.

\par So far, we provided an upper bound on generalization risk for unbiased empirical risk minimization in PU learning under SAR assumption. There is however no proof that this rate is optimal. In other words, is there another procedure that can learn a classifier $\widehat{g}$ that outperforms $\widehat{g}_{PU}$? A lower bound will help to answer this question.

\subsection{A lower bound on minimax risk}

In order to assess the optimality of the upper bound (\eqnref{upper}), we analyse and provide a lower bound on minimax risk.\\

The minimax risk is the risk of the classification procedure that performs best in the worst case. For any given estimate $\widehat{g}$, we recall that its generalization risk is measured as $\esp{\ell\prt{\widehat{g}, g^*}}$. The minimax risk is denoted $\Rsh$ and is defined as follows:
\begin{equation}\label{minimaxdef}
    \Rsh = \inf_{\widehat{g}\in \S}\brkt{\sup_{\P \in \Psh} \esp{\ell\prt{\widehat{g}, g^*}}}
\end{equation}
where the infimum is taken over the set of functions $\widehat{g}$ of $(X_i, \Lab_i)_{1\leq i\leq n}$ such that $\widehat{g}$ belongs to $\S$.\\

The bound \eqnref{upper-unif} is an obvious upper bound on the minimax risk. \thref{th2} establishes a lower bound on minimax risk for PU learning under SCAR assumption. Proposition \ref{propmm} extends it to SAR assumption.
%Then, as we want to assess its performance independently of the probability $\P$, we focus on the worst case scenario: $\sup_{\P \in \Psh}\esp{\ell\prt{\widehat{g}, g^*}}$. The minimax risk is then the minimum of this quantity among every function $\widehat{g}$ of $(X_i, \Lab_i)_{1\leq i\leq n}$ such that $\widehat{g}$ belongs to $\S$.

\begin{theorem}[Lower bound on minimax risk under SCAR assumption]\label{th2}~\\
Suppose that $V\geq 2$ and $n\, e_m \geq V$.
Let $\hp = \sqrt{\frac{V}{n\, e_m}}$.\\ Assuming $e(x)=e_m,\ \forall x\in \R^d$, there exists an absolute constant $\kappa_2>0$ such that:
\begin{enumerate}
    \item[\clabel{$C_1$}] if $h \geq \hp$:
\begin{equation}\label{lower1}
    \Rsh \geq \kappa_2 \frac{V-1}{h\, n\, e_m}\ ;
\end{equation}
    \item[\clabel{$C_2$}] if $h \leq \hp$:
\begin{equation}\label{lower2}
    \Rsh \geq \kappa_2 \sqrt{\frac{V-1}{n\, e_m}}\ .
\end{equation}
\end{enumerate}
\end{theorem}

\noindent\textbf{Remarks}

The lower bounds in \thref{th2} explicitly depends on $V$, $n$, $h$ and $e_m$. The cases \cref{$C_1$} and \cref{$C_2$} highlight a trade-off between the expected number of fully labeled instances (proportional to $n\,e_m$), the complexity of the model $V$ and the noise condition \cref{$A_2$} represented by $h$. The restriction of these results to standard classification setting ($e_m=1$) exactly matches existing results \citep[see][]{Massart2006Oct}. \thref{th2} moreover provides the influence of propensity $e_m$ in PU learning framework under SCAR assumption. As for the upper bound (\cf\  \thref{th1}), the lower bound \eqnref{lower1} is affected the same way with a degradation of order $1/e_m$ over the minimax rate when Massart noise condition \cref{$A_2$} is satisfied with $h$ high enough, in case \cref{$C_1$}. In this case, the lower bound rate almost matches the upper bound up to a logarithmic factor. In the second case \cref{$C_2$}, the lower bound is of order $\sqrt{V/n\,e_m}$ which exactly matches the rate of the upper bound in this regime. In this sense, $\widehat{g}_{PU}$ obtained through unbiased empirical risk minimization is almost optimal as it almost achieves the minimax convergence rates.\\

%Présenter différemment : à rajouter dans le théorème !
%mini preuve en annexe
 %This condition allows the rate in $\mathcal{O}\prt{\frac{V}{n\,h\,e_m}}$ to be valid as a lower bound on minimax risk. If the condition is violated, the minimax risk can still be lower bounded by a similar quantity with a rate in $\mathcal{O}\prt{\sqrt{\frac{V}{n\,e_m}}}$. Indeed, similarly as noted in the general minimax lower bound proof in \cite{Massart2006Oct}, $\Rsh \geq \Rm\prt{\hp,\S}$, where $\hp=\sqrt{\frac{V}{n\,e_m}}$. And now, as $\prt{n,e_m,V,\hp}$ satisfies condition \cref{H}, the same theorem can be applied to $\Rm\prt{\hp,\S}$, which yields a lower bound of order $\mathcal{O}\prt{\sqrt{\frac{V}{n\,e_m}}}$.\\

The detailed proof of \thref{th2} can be found in the Appendix~\ref{proofth2}. It makes use of similar arguments as for minimax lower bounds in standard classification setting. First, the expression of minimax risk is simplified by choosing a specific set of probabilities satisfying the noise conditions. Then Assouad lemma \citep{Yu1997} is applied to provide a lower bound on this simplified expression, where the singularity of PU learning mainly interferes.\\

To extend the result to SAR assumption, we need an extra condition:
%if an additional condition is satisfied. In the proof, the subset of probability distributions chosen is a family of discrete distributions whose support is a set of point $x_1, ..., x_V$ shattered by $\S$ achieving the minimum of propensity :
%\begin{equation}
%    \forall i \in \{1,...,V\},\ e(x_i)=e_m\ .
%\end{equation}
%Finding such a subset is obviously possible under SCAR assumption since $e$ is a constant function equal to $e_m$. However, in many realistic situations under SAR assumption, finding a subset of $V$ points satisfying these two conditions is still possible. Here is a sufficient condition for lower bound $\ref{lower}$ to be still satisfied under SAR assumption :
\begin{enumerate}
\item[\clabel{$A_3$}] $\forall\,\varepsilon>0,\ \exists\, (x_1,...,x_V)\in \prt{\R^d}^V$ \text{scattered by }$\S$ and such that:
\begin{equation*}
\sup_{i \in \{1,...,V\}} e(x_i)\leq e_m + \varepsilon\ .
\end{equation*}
\end{enumerate}
This assumption is technical. It is used in the first step of the proof of the minimax lower bound as it allows us to choose a convenient family of discrete probability distributions satisfying the noise assumptions. Assumption \cref{$A_3$} is fulfilled in natural situations, for example, when $e(\cdot)$ is continuous and $\S$ is the set of linear classifiers in $\R^d$.

\begin{proposition}[Lower bound on minimax risk under SAR assumption]\label{propmm}~\\
\thref{th2} extends to SAR assumption if the propensity $e(\cdot)$ greater than $e_m>0$ and if assumption \cref{$A_3$} is satisfied. 
\end{proposition}

The proof of the above proposition can be found in Appendix \ref{proofprop}. The same remarks as for \thref{th2} remain valid under SAR assumption when assumption \cref{$A_3$} is satisfied. In particular, in regimes \cref{$C_1$} and \cref{$C_2$}, the minimax rate still matches the upper bound rate \eqnref{upper} up to the logarithmic factor.

\section{Conclusion}\label{conclusion}
In this paper, we provided a theoretical study of PU learning under SAR assumption, \textit{i.e.} when the probability for an instance to be labeled depends on its covariates. Assuming partial knowledge on the propensity, a consistent classifier can be identified by minimizing a conveniently weighted empirical risk. We established a general non-asymptotic upper bound on the excess risk that naturally extends known risk bounds in standard classification setting. By providing a minimax lower bound, we then showed that the convergence rates are optimal up to a logarithmic term. Besides, these bounds explicitly quantify the difficulty of the PU learning task in terms of the propensity. Finally, both results show that under noise assumptions, fast rates can be achieved.

As a future perspective, it would be interesting to study how some assumptions made on the propensity could be relaxed. In particular, future work could assess whether or not the theoretical guaranties proved in this paper still hold when the propensity is estimated. Likewise, we may wonder if these results could be extended if the lower bound on the propensity only holds with high probability.
To bridge the gap between these theoretical results and practical PU learning methodologies adapted to SAR assumption, several challenges remain open. For instance, the estimation of the propensity is a difficult problem.
Besides, minimizing the unbiased empirical risk in PU learning based on $0-1$ loss requires solving computationally difficult combinatorial optimization problems. The use of convex loss functions would facilitate the optimization. Then, it could be interesting to study how our theoretical results extend to such case.

% Acknowledgements should go at the end, before appendices and references

\acks{This work was carried out within the framework of the partnership between Stellantis and the OpenLab AI with the financial support of the ANRT for the CIFRE contract n°2019/1131.}


\newpage

\appendix
\section{Proof of \thref{th1}}
\label{proofth1}
%The proof relies on a  general result from \cite{Massart2006Oct} stating an upper bound for the (vérifier vocab)expected risk excess of empirical risk minimizers when the loss function is at values in $[0,1]$, that we extend to the case of a loss function at values in $[a;b]$ ($a<b$). Applying this general result to the PU learning loss function $\empsar$ (\cf\  ) leads to an upper bound.

The proof is organized as follows. We first state a general upper bound result for empirical risk minimizers 
%which extends Massart and Nédélec's theorem (\citeyear[Theorem~2]{Massart2006Oct}) 
adapted to the case where the loss function takes values in an arbitrary interval $[a,b]$ with $a<b$ (\cf\  \ssref{GRUBERM}). Then, we show that the PU learning loss function satisfies the assumptions of this general result (\cf\  \ssref{assumptions}). Finally, we deduce the upper bound as the solution of a fixed point equation (\cf\  \ssref{result1}).
%Actually, as our empirical risk function is not in $[0,1]$ (which is a requirement in \cite{Massart2006Oct}), we first show that the result can be re-written for a risk $\gamma$ with values in $[a,b]$ ($a<b$). This will be an opportunity to recall the general result from \cite{Massart2006Oct}.

\subsection{General risk upper bound on empirical risk minimizers}\label{GRUBERM}
We begin by stating a general upper bound theorem for empirical risk minimizers.
\begin{theorem}[General upper bound for empirical risk minimizers]
\label{th3}~\\
Let $\emp$ be an unbiased loss function with values in $[a,b]$, $\Emp$ the empirical risk, $\Empc$ the centered empirical risk. Let $g^*$ denote the Bayes classifier and let $\widehat{g}$ be a minimizer of the empirical risk over a class $\S$ for which we assume separability condition \cref{$A_1$}. Let $\ell$ denote the  excess risk. We assume that:
\begin{enumerate}
    \item[\clabel{$B_1$}] there exists a positive and symmetric function $d$ such that for any couple of classifiers $(g,g^\prime)$:
    \begin{equation*}
         Var\brkt{\emp\prt{g^\prime, \prt{X,\Lab}}-\emp\prt{g, \prt{X,\Lab}}} \leq d^2\prt{g^\prime, g};
    \end{equation*}
    \item[\clabel{$B_2$}] there exists an non-decreasing function $w$ continuous on $\R_+$, such that $x\mapsto \frac{w(x)}{x}$ is non-increasing on $\R_+^*$, with $w(\sqrt{b-a})\geq b-a$ and ensuring for any classifier $g$:
    \begin{equation*}
        d(g^*,g)\leq w\prt{\sqrt{\ell(g^*,g)}};
    \end{equation*}
    \item[\clabel{$B_3$}] there exists an non-decreasing function $\Phi$ continuous on $\R_+$, such that $x\mapsto \frac{\Phi(x)}{x}$ is non-increasing with $\Phi(b-a)\geq b-a$ and ensuring:
    \begin{equation*}
        \forall h \in \S^\prime, \sqrt{n}\, \esp{\sup_{g \in \S^\prime, d\prt{g,h}\leq \sigma} \Empc(h)-\Empc(g)} \leq \Phi(\sigma).
    \end{equation*}
    for every positive $\sigma$ such that $\Phi(\sigma)\leq \sqrt{n}\frac{\sigma^2}{b-a}$, where $\G^\prime$ comes from separability condition \cref{$A_1$}.
\end{enumerate}
    Then there exists a absolute constant $\kappa>0$ such that:
    \begin{equation}\label{bound}
        \esp{\ell(g^*, \widehat{g})} \leq \kappa\, \varepsilon_*^2,
    \end{equation}
    where $\varepsilon_*$ is the unique positive solution of following equation:
    \begin{equation}\label{fixed-point}
        \sqrt{n}\, \varepsilon_*^2 = \Phi\prt{w\prt{\varepsilon_*}}.
    \end{equation}
\end{theorem}

\begin{proof}
The above result follows from the application of Massart and Nédélec's theorem (\citeyear[Theorem~2]{Massart2006Oct}) using the re-scaled risk $\tilde{\gamma}=\frac{\gamma-a}{b-a}$ and the functions $\tilde{d}(g,g^\prime) = \frac{d(g,g^\prime)}{b-a}$, $\tilde{w}(x) = \frac{1}{b-a}\, w\prt{x\sqrt{b-a}}$ and $\tilde{\Phi}(x)=\frac{1}{b-a}\Phi\prt{\prt{b-a}x}$. This leads to the upper bound in \eqnref{bound} solution of \eqnref{fixed-point}.
\end{proof}
Note that now, contrary to Massart and Nédélec's original result, \cref{$B_2$} and \cref{$B_3$} explicitly involve the length of the interval $[a,b]$. This will be accounted for in our proof.

\subsection{Verification of assumptions of \thref{th3} in PU learning setting} \label{assumptions}
We first recall the definition and the main property of PU learning loss function as defined in \ssref{sarloss}. We then exhibit three functions $d$, $w$, $\Phi$ fulfilling conditions \cref{$B_1$}, \cref{$B_2$} and \cref{$B_3$}. Hence we show that the general upper bound result (\ie\ \thref{th3}) can be applied in PU learning context.\\

In the context of PU learning under SAR assumption, we recall that the loss function $\empsar$ is defined as follows:
\begin{equation*}
    \empsar\prt{g,\prt{X,\Lab}} = \frac{\ind{\Lab=1}}{e(X)}\prt{2\, \ind{g(X)\neq 1} - 1} + \ind{g(X) \neq 0}
\end{equation*}
where $e(x)=\P\prt{\Lab=1\, \vert \, \Class=1,X=x}$ is the propensity assumed to be known for labeled observations. Knowing that the propensity greater than $e_m>0$, the loss function is then at values in $\brkt{1-\frac{1}{e_m}, \frac{1}{e_m}}$, interval of length: 
\begin{equation}\label{Ce}
C_e = \frac{2}{e_m}-1\ .
\end{equation}

We have seen that this empirical risk is an unbiased estimate of the true risk (\cf\  \eqnref{sar-unbiased}):
\begin{equation*}
    \esp{\empsar\prt{g,\prt{X,\Lab}}} = \P\prt{g(X)\neq \Class}\ .
\end{equation*}
In order to apply the general upper bound theorem (\thref{th3}) to PU learning risk minimizer, we need to identify three functions $d$, $w$, $\Phi$ satisfying conditions \cref{$B_1$}, \cref{$B_2$} and \cref{$B_3$}. These functions are crucial since the upper bound is the solution of a fixed point equation involving them. The choice of functions $d$, $w$ and $\Phi$ will be a consequence of Propositions \ref{prop1}, \ref{prop2} and \ref{prop3}.

\begin{proposition}\label{prop1}
For any pair of classifiers $(g,g^\prime)$:
\begin{equation*}
Var\brkt{\empsar\prt{g^\prime, \prt{X,\Lab}}-\empsar\prt{g, \prt{X,\Lab}}} \leq 2\, C_e\, \esp{\abs{g(X)-g^\prime(X)}^2}\ ,
\end{equation*}
where $C_e$ is given by \eqnref{Ce}.
\end{proposition}
\paragraph*{Remark:} A direct consequence of the above proposition is that the function $d$ defined as:
\begin{equation}\label{d-def}
    d(g,g^\prime) = \sqrt{2 C_e} \sqrt{\esp{\abs{g(X)-g^\prime(X)}^2}}
\end{equation}
satisfies condition \cref{$B_1$}.\\
\begin{proof}
We first provide an upper bound on the variance of increments of $\empsar$:
\begin{subequations}
\begin{align}
Var \brkt{\empsar\prt{g} - \empsar\prt{g^\prime}}
 &\leq \esp{\prt{\empsar\prt{g} - \empsar\prt{g^\prime}}^2} \nonumber\\
&= \esp{\prt{g(X)-g^\prime(X)}^2 \prt{1-\frac{2\mathbf{1}_{\Lab=1}}{e(X)}}^2} \nonumber\\
&= \esp{\prt{g(X)-g^\prime(X)}^2 \esp{\prt{1-\frac{2\mathbf{1}_{\Lab=1}}{e(X)}}^2\vert X}}\nonumber\\
&= \esp{\prt{g(X)-g^\prime(X)}^2 \prt{1+4\eta(X)\frac{1-e(X)}{e(X)}}}\label{eqd-a}\\
& \leq \prt{1 + 4 \frac{1-e_m}{e_m}}\esp{\prt{g(X)-g^\prime(X)}^2} \label{eqd-b}\\
& \leq 2\, C_e\, \esp{\prt{g(X)-g^\prime(X)}^2}\ . \nonumber\end{align}
\end{subequations}
%\textcolor{teal}{First, Eq.~\eqref{eqd-b}, variance is bounded above by the second moment.}\\
%In the previous computation, we upper bound the variance of increments by the expectation of its square \eqref{eqd-b}, which yields \eqref{eqd-c} using the expression of $\empsar$. 
We then use the fact that $\esp{\ind{\Lab=1}\vert X} = \esp{\eta(X)\, e(X)}$ to get \eqnref{eqd-a}. And \eqnref{eqd-b} results from the fact that $\eta(X)$ is less than $1$ and $e(X)$ is greater than $e_m$.
%we get an expression involving the regression function $\eta$ and the propensity \ref{eqd-e}.
%\textcolor{teal}{Note that $\esp{\ind{\Lab=1}\vert X} = \esp{\eta(X) e(X)}$ with $eta(X)<1$ and $e(X)\leq e_m$. Finally, let us recall that $C_e=\frac{2}{e_m}-1$ to obtain Eq.~\eqref{eqd-g}}.
%Indeed, using the definition of propensity: $\esp{\ind{\Lab=1}\vert X} = \esp{\eta(X) e(X)}$. 
%Finally, we upper bound $\eta$ by $1$ and use the fact that $e$ is lower bounded by $e_m$ \eqref{eqd-f}. The constant in front of the expectation can be upper bounded by $2\, C_e$ \eqref{eqd-g}. We recall that $C_e=\frac{2}{e_m}-1$.
\end{proof}

\begin{proposition}\label{prop2}
For any classifier $g$:
\begin{equation*}
d(g,g^*) \leq \sqrt{\frac{2\, C_e}{h}} \sqrt{\ell(g,g^*)}\ .
\end{equation*}
for $d$ defined in \eqnref{d-def}.
\end{proposition}
\paragraph*{Remark:} As a consequence, the function $w$ defined as:
\begin{equation}\label{w}
    w(x) = \sqrt{\frac{2\, C_e}{h}}\, x\ .
\end{equation}
satisfies Assumption \cref{$B_2$}: $w$ is continuous on $\R_+$, non-decreasing, such that $x \mapsto \frac{w(x)}{x}$ is non-increasing and $w\prt{\sqrt{C_e}}\geq C_e$, and such that:
\begin{equation}\label{d}
    d(g^*,g)\leq w\prt{\sqrt{\ell\prt{g^*, g}}}\ .
\end{equation}

Let $\hp=\sqrt{V/n\, e_m}$. Note that the function
\begin{equation}\label{w0}
w_0(x) = \sqrt{2\, C_e} \vee x\,\sqrt{2\, C_e / \hp}
\end{equation}
also satisfies assumption \cref{$B_2$}.

\begin{proof}
The excess risk can be expressed in terms of $\eta(X)$ as follows:
\begin{equation}
\begin{split}
    \ell(g,g^*) &= \P\prt{g(X)\neq \Class} - \P\prt{g^*(X)\neq \Class} \\
%    &= \esp{\prt{\Class-s(X)}^2 - \prt{\Class-g^*(X)}^2} \label{eql-a}  \\
%    &= \esp{\prt{g^*(X)-s(X)} \prt{2\, \Class - s(X) - g^*(X)}} \nonumber\\
%    &= \esp{\prt{g^*(X)-s(X)} \prt{2\, \eta(X) - 1}} \label{eql-b}\\
    &= \esp{\abs{g(X)-g^*(X)}^2 \abs{2\, \eta(X) - 1}} \label{eql}\ .
\end{split}
%\eqnref{eql-a} is obtained by noting that  $\ind{s(X)\neq \Class}=(s(X)- \Class)^2$ and $\ind{g^*(X)\neq \Class}=(g^*(X)- \Class)^2$, \eqnref{eql-b} is due to the fact that $g^*(X) + s(X)=1$ whenever $g^*(X)$ and $s(X)$ are different. It then remains to see that the function in the expectation is always equal to its absolute value using the definition of Bayes classifier $g^*(X)=\ind{\eta(X)\geq \frac{1}{2}}$ (cf \eqnref{eql-c}).
\end{equation}

%Then, using Massart noise condition (\ref{noise-massart}) on \eqref{eql-f}, we can lower bound the excess risk with a quantity that can be linked to the previous pseudo-distance $d$ as defined in \eqref{d}:
Then, using the noise assumption \cref{$A_2$} and the definition of $d$ \eqnref{d-def}, we have the following lower bound on the excess risk:
\begin{equation*}
\begin{split}
    \ell(g,g^*) &= \esp{\prt{g(X)-g^*(X)}^2 \abs{2\, \eta(X) - 1}}\\
    & \geq h\, \esp{\prt{g(X)-g^*(X)}^2} \\
    & = \frac{h}{2\, C_e} d^2(g, g^*)\ .
\end{split}
\end{equation*}
Taking the square root on both side finishes the proof.
\end{proof}

The next proposition states the existence of $\Phi$ fulfilling \cref{$B_3$}. We recall that the subset $\S^\prime \subset \S$ is given by the separability assumption \cref{$A_1$} and that the constant $C_e$ is defined in \eqnref{Ce}.
\begin{proposition}\label{prop3}
Assume $\S$ has finite VC dimension $V$ and $\S^\prime$ is given by separability assumption \cref{$A_1$}. There exists a absolute constant $K\geq 1$ such that the function $\Phi$ defined as
\begin{equation}\label{phi}
  \Phi(\sigma) = K \sigma \sqrt{V \brkt{1 + \log\prt{\frac{C_e}{\sigma}\vee 1}}}  
\end{equation}
satisfies:
\begin{equation*}
\sqrt{n}\, \esp{\sup_{g \in \S^\prime, d\prt{g,h}\leq \sigma} \Empsarc(g_0)-\Empsarc(g)} \leq \Phi(\sigma)
\end{equation*}
for all $g_0 \in \S^\prime$ and for every $\sigma$ such that $\Phi(\sigma)\leq \sqrt{n}\, \frac{\sigma^2}{C_e}$.
\end{proposition}
\begin{proof}
We consider a fixed $g_0 \in \S^\prime$ along the proof and use the notation:
\begin{equation*}
    W = \sup_{g \in \S^\prime, d\prt{g,g_0}\leq \sigma} \Empsarc(g_0)-\Empsarc(g)\ .
\end{equation*}
The main steps of the proof are: \nameref{par1} rewrite $W$ as the supremum of an empirical process over a class of functions; \nameref{par2} split the expression of $W$ in two terms depending on the sign of $(g_0(x)-g(x))$ ($W^+$ and $W^-$) that will be processed similarly and independently; \nameref{par3} provide an upper bound on $\esp{W^+}$ using a symetrization principle \citep[\cf][]{Bousquet2003Feb}; \nameref{par4} apply a chaining inequality and Haussler bound \citep{Bousquet2003Feb, Massart2006Oct}; \nameref{par5} a few calculations finish the proof. This proof uses the notion of entropy metrics: the definition and some useful properties are recalled in Appendix \ref{uem}.\\

\paragraph*{(i)}\label{par1} We start by rewriting the expression inside the supremum in $W$:
\begin{equation*}
    \begin{split}
        \Empsarc(g_0)&-\Empsarc(g) = \Empsar(g_0)-\Empsar(g) - \esp{\Empsar(g_0)-\Empsar(g)}\\
        &=\frac{1}{n}\sum_{i=1}^n \prt{\empsar\prt{g_0,(X_i,\Lab_i)} - \empsar\prt{g,(X_i,\Lab_i)}} - \esp{\Empsar(g_0) - \Empsar(g)}\\
        &= \frac{1}{n}\sum_{i=1}^n \prt{g_0(X_i)-g(X_i)}\prt{\frac{2\, \ind{\Lab_i=1}}{e(X_i)} - 1} - \esp{\prt{g(X)-g_0(X)}\prt{\frac{2\, \ind{\Lab=1}}{e(X)} - 1}}\\
        &= \prt{\P_n - \P}(f_g),
    \end{split}
\end{equation*}
%In the last line, $\P_n$ denotes the empirical measure ($\P_n = \frac{1}{n}\sum_{i=1}^n$) and $f_g$ is the function:
where $\P_n f_g$ and $\P f_g$ denote respectively the empirical mean and the expectation of the function $f_g$:
\begin{equation*}
    f_g:(x,s)\mapsto\prt{g_0(x)-g(x)}\prt{\frac{2\, \ind{s=1}}{e(x)} - 1}.
\end{equation*}
Hence, denoting $\mathcal{F}(\sigma)=\crochets{f_g: g \in \S^\prime,\ d(g_0,g)\leq \sigma}$, we can write $W$ as the supremum of the empirical process $(\P_n-\P)(\cdot)$ over the set of functions $\mathcal{F}(\sigma)$:
\begin{equation}\label{prop3-step1}
    W = \sup_{f \in \mathcal{F}(\sigma)} (\P_n-\P)(f).
\end{equation}

\paragraph*{(ii)}\label{par2} For any $g \in \S^\prime$, we can decompose $f_g$ depending on the sign of $\prt{g_0(x)-g(x)}$:
\begin{equation*}
    f_g\prt{x,s} = \prt{\frac{2\, \ind{s=1}}{e(x)} - 1} \ind{g(x) > g_0(x)} - \prt{\frac{2\, \ind{s=1}}{e(x)} - 1} \ind{g_0(x) > g(x)}\ .
\end{equation*}
Then, introducing the following classes of functions
\begin{equation*}
    \begin{split}
        & \mathcal{F}^+(\sigma) = \crochets{f : \R^d \times \crochets{0,1} \rightarrow \R,\ \exists\ g \in \S^\prime,\ f(x,s) = \brkt{\frac{2 \ind{s=1}}{e(X)}-1}\ind{g(x)>g_0(x)} \mbox{ , } d(g_0,g)\leq \sigma}\\
& \mathcal{F}^-(\sigma) = \crochets{f : \R^d \times \crochets{0,1} \rightarrow \R,\ \exists\ g \in \S^\prime,\ f(x,s) = \brkt{\frac{2 \ind{s=1}}{e(X)}-1}\ind{g(x)<g_0(x)} \mbox{ , } d(g_0,g)\leq \sigma}\\
    \end{split}
\end{equation*}
and the corresponding suprema
\begin{equation*}
    \begin{split}
        & W^+ = \sup_{f \in \mathcal{F}^+(\sigma)} (\P_n-\P)(f) \\
        & W^- = \sup_{f \in \mathcal{F}^-(\sigma)} (\P-\P_n)(f),
    \end{split}
\end{equation*}
we decompose $\esp{W}$ as follows:
\begin{equation}\label{prop3-step2}
   \esp{W} \leq \esp{W^+} +  \esp{W^-}.
\end{equation}
%Applying symmetrization property (cf ...), we can bound each term by a Rademacher complexity:
%\begin{equation*}
%\esp{\sup_{t \in \S^\prime, d\prt{u,t}\leq \sigma} \Empsarc(h)-\Empsarc(g)} \leq \frac{2}{n}\esp{\sup_{f \in \mathcal{F}^+(\sigma)} \sum_{i=1}^n \varepsilon_i f(X_i,Y_i)} + \frac{2}{n}\esp{\sup_{f \in \mathcal{F}^-(\sigma)} \sum_{i=1}^n \varepsilon_i f(X_i,Y_i)}
%\end{equation*}
We now process both terms separately focusing on $W^+$ (the proof for the other term is almost identical). 

%We will rely on several properties on empirical processes recalled in \cite{Massart2006Oct}. The difference here is that the class of functions we have is not a class of indicator functions which prevents us from applying directly lemma A.3 from \cite{Massart2006Oct}. Still, using similar arguments, we obtain an analogous result adapted to our case.\\

\paragraph*{(iii)}\label{par3} We first apply a symetrization principle to provide an upper bound on $\esp{W^+}$ depending on a Rademacher average \citep[\cf][]{Bousquet2003Feb}:
\begin{equation}\label{prop3-step3}
    \esp{W^+} \leq \frac{2}{n}\,\esp{\sup_{f \in \mathcal{F}^+(\sigma)} \sum_{i=1}^n \varepsilon_i f(X_i,\Lab_i)}
\end{equation}
where $(\varepsilon_i)_{1\leq i\leq n}$ are i.i.d. Rademacher variables (\textit{i.e.} $\P\prt{\varepsilon_i = 1} = \P\prt{\varepsilon_i = -1}=\frac{1}{2}$).\\

\paragraph*{(iv)}\label{par4} Let $\delta^2 = \sup_{f \in \mathcal{F}^+(\sigma)}\P_n\prt{f^2} \vee \sigma^2$. We apply a chaining inequality (lemma A.2, \citealt{Massart2006Oct}) which gives us the following inequality:
\begin{equation}\label{eqw-1}
    \esp{W^+} \leq \frac{6}{\sqrt{n}} \esp{\delta\, \sum_{j=0}^{+\infty} 2^{-j} \sqrt{H\prt{2^{-j-1} \delta, \mathcal{F}_+(\sigma)}}}
\end{equation}
where $H$ is the universal entropy metric (\cf\  Appendix \ref{uem}).

Let $\mathcal{A}_+ = \crochets{\ind{g(x)>g_0(x)}, g \in \S^\prime}$, which can be consider as a set of classifiers and has VC dimension $V$ at most. Using the fact that $H\prt{\cdot, \mathcal{F}_+(\sigma)}$ is non-increasing (\cf\  Proposition \ref{decreasing}), we have $\forall j \geq 0$:
\begin{equation*}
   H\prt{2^{-j-1} \delta, \mathcal{F}_+(\sigma)} \leq H\prt{2^{-j-1} \sigma, \mathcal{F}_+(\sigma)} \label{eqh-a} \ . 
\end{equation*}

Applying Proposition \ref{contraction}, we obtain the following upper bound on the entropy of $\mathcal{F}_+(\sigma)$ in terms of the entropy of $\mathcal{A}_+$:
\begin{equation*}
    H\prt{2^{-j-1} \delta, \mathcal{F}_+(\sigma)} \leq H\prt{2^{-j-1} \frac{\sigma}{C_e},\ \mathcal{A}_+(\sigma)}\ . 
\end{equation*}

We are then in position to apply Haussler bound (Proposition \ref{haussler}), to get an upper bound on the entropy in terms of the VC dimension of $\mathcal{A}_+$ which is no more than $V$:

\begin{equation}\label{complexity}
   H\prt{2^{-j-1} \delta, \mathcal{F}_+(\sigma)} \leq \kappa\, V \prt{1+\log\prt{2^{j+1}\frac{C_e}{\sigma}\vee 1}} 
\end{equation}
for some absolute constant $\kappa>1$.

\paragraph*{(v)}\label{par5} Injecting \eqnref{complexity} in \eqnref{eqw-1}, we get:
\begin{subequations}
\begin{align}
    \esp{W^+} &\leq 6\sqrt{\frac{\kappa\, V}{n}} \brkt{\sum_{j=0}^{+\infty}2^{-j} \sqrt{1+\log\prt{2^{j+1}\frac{C_e}{\sigma}\vee 1}}}\, \esp{\delta} \notag \\
    &\leq C(\sigma)\, \sqrt{\frac{V}{n}} \esp{\delta} \label{eq-cs-a}\\
    &\leq C(\sigma)\, \sqrt{\frac{V}{n}} \sqrt{\esp{\delta^2}} \label{eq-cs-b}
\end{align}
\end{subequations}
where $C(\sigma)=12\, \prt{1+\log(2)}\sqrt{\kappa}\, \sqrt{1+\log\prt{\frac{C_e}{\sigma}\vee 1}}$. \eqref{eq-cs-a} is a consequence of technical Lemma \ref{lemma1} in Appendix \ref{technical}, \eqref{eq-cs-b} follows from Cauchy-Schwartz inequality.

Now, we provide an upper bound on $\esp{\delta^2}$ in terms of $\esp{W^+}$:
\begin{align}
    \esp{\delta^2} &\leq \sigma^2 + \esp{\sup_{f \in \mathcal{F}_+(\sigma)} \P_n\prt{f^2}} \nonumber \\
    & \leq \sigma^2 + C_e\, \esp{\sup_{f \in \mathcal{F}_+(\sigma)} \P_n\prt{f}} \nonumber\\
    & \leq \sigma^2 + C_e\, \esp{\sup_{f \in \mathcal{F}_+(\sigma)} (\P_n-\P)\prt{f}} + C_e\, \, \sup_{f \in \mathcal{F}_+(\sigma)} \P(f) \label{upper-delta}
    % & \leq C_e\, \esp{W^+} + 2\sigma^2 
\end{align}

Let $f \in \mathcal{F}_+(\sigma)$ and define $g \in \S^\prime$ such that $f(x,s) = \brkt{\frac{2\ind{s=1}}{e(x)}-1} \ind{g_0(x)>g(x)}$ (and $d(g_0, g)\leq \sigma$). We have:
\begin{align*}
    \P(f) %&= \esp{\prt{\frac{2\ind{\Lab=1}}{e(X)}-1} \ind{g_0(X)>g(X)}}\\
    &= \esp{\esp{\frac{2\ind{\Lab=1}}{e(X)}-1 \vert X} \ind{g_0(X)>g(X)}} \\
    &= \esp{\prt{2 \eta(X) - 1} \ind{g_0(X)>g(X)}} \\
    &\leq \esp{\abs{g_0(X) - g(X)}^2} \\
    &= \frac{d^2(g_0, g)}{2\, C_e} \\
    &\leq \frac{\sigma^2}{2\, C_e} \ .
\end{align*}
using \eqnref{d-def} and the definition of $\mathcal{F}_+(\sigma)$.
Note that the above upper bound does not depend on $f\in \mathcal{F}_+(\sigma)$. Hence, we can use it in \eqnref{upper-delta} to obtain:
\begin{equation}\label{upper-delta2}
    \esp{\delta^2} \leq C_e\, \esp{W^+} + \frac{3}{2}\sigma^2 
\end{equation}

Hence, coming back to $\esp{W^+}$:
\begin{equation*}
\begin{split}
    \esp{W^+} &\leq C(\sigma)\, \sqrt{\frac{V}{n}}\, \sqrt{C_e\, \esp{W^+} + \frac{3}{2}\sigma^2}\ .
\end{split}
\end{equation*}
Taking the square on both sides and solving the second order inequation in $\esp{W^+}$ yields:
\begin{equation*}
    \esp{W^+} \leq \frac{1}{2}C(\sigma)\sqrt{\frac{V}{n}}\prt{C(\sigma)\, C_e \sqrt{\frac{V}{n}} + \sqrt{\frac{C(\sigma)^2\, C_e^2\, V}{n}  + 6 \sigma^2}}\ .
\end{equation*}

Therefore, whenever $\sigma\geq C(\sigma)\, C_e\, \sqrt{\frac{V}{n}} $:
\begin{equation*}
    \sqrt{n}\, \esp{W^+} \leq 2\, \sigma\, C(\sigma) \sqrt{V}\ .
\end{equation*}

We can prove a similar upper bound on $\esp{W^-}$. If we define $\Phi(\sigma) = 4\, \sigma\, C(\sigma) \sqrt{V}$, for all $\sigma$ such that $\Phi(\sigma)\leq \sqrt{n}\,\frac{\sigma^2}{C_e}$ (condition of Proposition \ref{prop3}):
\begin{equation*}
     \sigma\geq C(\sigma)\, C_e\sqrt{\frac{V}{n}}  \ .
\end{equation*}
Hence, we have the desired upper bound on $\esp{W}$:
\begin{equation*}
    \sqrt{n}\,\esp{W} \leq \Phi(\sigma)\ .
\end{equation*}
Note that the constant $K=4\,C(\sigma)$ is greater than $1$.
\end{proof}

\subsection{Upper bounds on the risk}\label{result1}
In the previous subsection, we checked that \thref{th3} can be applied to PU learning under SAR assumption. Hence, the upper bound on risk excess $\varepsilon_*^2$ is the unique solution to the fixed point equation:
\begin{equation}\label{fixed-point-eqn}
    \sqrt{n}\, \varepsilon_*^2 = \Phi\prt{w(\varepsilon_*)}
\end{equation}
where $w$ is given in \eqnref{w} (or $w_0$ in \eqnref{w0}) and $\Phi$ in \eqnref{phi}.
\begin{equation*}
    w(x) = \sqrt{\frac{2\, C_e}{h}}\, x\ ,
\end{equation*}
\begin{equation*}
    w_0(x) = \sqrt{2\, C_e} \vee x\,\sqrt{\frac{2\, C_e}{\hp}}\ ,
\end{equation*}
\begin{equation*}
    \Phi(\sigma) = K \sigma \sqrt{V \brkt{1 + \log\prt{\frac{C_e}{\sigma}\vee 1}}}\ .
\end{equation*}

We cannot explicitly solve this equation, but we can provide an upper bound on the solution which is enough to complete the proof of \thref{th1}. The choice of $w$ as \eqnref{w} or \eqnref{w0} leads to two different upper bounds (Subsections \ref{case1} and \ref{case2}) that together complete the proof of \thref{th1}.

\subsubsection{First case}\label{case1}

Using the known definitions of $w$ in \eqnref{w} and $\Phi$ in \eqnref{phi}, \eqnref{fixed-point-eqn} can be rewritten as:
\begin{equation}\label{eqfp}
    \sqrt{n}\, \varepsilon_*^2 = K\, \varepsilon_*\, \sqrt{\frac{2\, C_e}{h}}\, \sqrt{V\, \brkt{1+\log\prt{\frac{\sqrt{C_e\, h}}{\sqrt{2}\varepsilon_*}\vee 1}}}
\end{equation}
Because the $log$ term is always non negative and $K\geq 1$, we get:
\begin{equation*}
    \varepsilon_* \geq \sqrt{\frac{2\, C_e\, V}{n\, h}}\ .
\end{equation*}

Using this on the logarithmic term, we obtain the following upper bound on $\varepsilon_*$:
\begin{equation*}
\begin{split}
    \varepsilon_* &\leq K \sqrt{\frac{2\, C_e\, V}{n\, h}}\, \sqrt{1+\log\prt{\frac{\sqrt{n}\, h}{2\, \sqrt{V}}\vee 1}}\\
    &\leq K \sqrt{\frac{2\, C_e\, V}{n\, h}}\, \sqrt{1+\log\prt{\frac{n\, h^2}{V}\vee 1}}
\end{split}
\end{equation*}
Note that $C_e \leq \frac{2}{e_m}$. Finally, we get the desired result:
\begin{equation*}
    \varepsilon_*^2 \leq 4\, K^2\, \frac{V}{n\, h\, e_m}\brkt{1+\log\prt{\frac{n\, h^2}{V}\vee 1}}\ .
\end{equation*}
\begin{flushright}$\blacksquare$\end{flushright}

\subsubsection{Second case}\label{case2}
We now consider \eqnref{fixed-point-eqn} where $w$ is given by \eqnref{w0}. Note that the logarithmic term is necessarily $0$. If we assume that the solution $\varepsilon_*$ of \eqnref{fixed-point-eqn} satisfies $\varepsilon_* \geq \sqrt{\hp}$, then $w(x)=\varepsilon_*\, \sqrt{\frac{2\, C_e}{\hp}}$. We obtain:
\begin{equation*}
    \varepsilon_*^2 \leq 4\, K^2\, \sqrt{\frac{V}{n\, e_m}}\ .
\end{equation*}
Else, $\varepsilon_* \leq \sqrt{\hp}$ which implies that
\begin{equation*}
    \varepsilon_*^2 \leq \hp=\sqrt{\frac{V}{n\, e_m}}\ .
\end{equation*}
Both bounds provide the same convergence rate. 

Paragraphs \ref{case1} and \ref{case2} together complete the proof of \thref{th1}.

\section{Proof of minimax lower bounds}
We remind the reader that the minimax risk is defined as:
\begin{equation*}
    \Rsh = \inf_{\widehat{g}\in \S}\brkt{\sup_{\P \in \Psh} \esp{\ell\prt{\widehat{g}, g^*}}}\ .
\end{equation*}
The lower bound on minimax risk under is proved in \ssref{proofth2} for SCAR assumption (\cf\  \thref{th2}) and in \ssref{proofprop} for SAR assumption (\cf\  Proposition \ref{propmm}).

\subsection{Under SCAR assumption (proof of \thref{th2})}\label{proofth2}
The proof consists in exhibiting a finite subset of family distributions on which the excess risk is worst. It is organised as follows: \nameref{parb1} we provide a lower bound on the minimax risk expression by restricting ourselves to this subset of distributions; \nameref{parb2} we use Massart noise condition and simplify the remaining expression; \nameref{parb3} the application of Assouad lemma finishes the proof.\\

\paragraph*{(i)} \label{parb1} We start by introducing a family of probability distributions which plainly exploits the noise condition \cref{$A_2$}. Let $x_1,...,x_V$ be $V$ points of $\R^d$ shattered by $\S$. This is possible because the VC dimension of $\S$ is $V$. For some parameter $p<\frac{1}{V-1}$, we define a discrete probability distribution on $\{x_1,...,x_V\} \subset \R^d$ verifying:
\begin{equation*}
    \P\prt{X=x_i}=p\ \ \forall\, i\leq V-1 \text{ and }\P\prt{X=x_V}=1-p\,\prt{V-1}\ .
\end{equation*}
For some binary vector $b \in \crochets{0,1}^{V-1}$, we consider $\P_b$ the probability distribution such that:
\begin{equation*}
    \forall\, 1\leq i\leq V-1,\ \P_b\prt{\Class=1\,\vert\,X=x_i} = \frac{1}{2}\brkt{1+\prt{2\, b_i-1}\, h}
\end{equation*}
for $h>0$. We can consider by default that each point in $\R^d\setminus\crochets{x_1,...,x_{V-1}}$ has class $0$ almost surely. This has no incidence on the rest of the proof. Moreover:
\begin{equation*}
    \P_b\prt{\Lab=1 \, \vert\, X=x_i, \Class=y} = y\, e(x_i)
\end{equation*}
following the definition of propensity.

Hence, $\prt{\P_b}_{b\in \crochets{0,1}^{V-1}}$ defines a family of distributions on $(X,\Lab)$ that satisfies Massart noise condition \cref{$A_2$} at its limits: the regression function $\abs{2\eta(x_i)-1}$ equals $h$ for every $i \in \crochets{1,...V-1}$. Furthermore, for every $b\, \in \crochets{0,1}^{V-1}$, the Bayes classifier $g^*_b$ is known:
\begin{equation*}
    \forall\, 1\leq i\leq V-1,\ g^*_b(x_i)=b_i \ .
\end{equation*}
As $(x_1,...,x_V)$ is shattered by $\S$, $g^*_b$ necessarily belongs to $\S$.\\ Hence, $\prt{\P_b}_{b\in \crochets{0,1}^{V-1}} \subset \Psh$ and therefore:
\begin{equation*}
    \Rsh \geq \inf_{\widehat{g}\in \S}\brkt{\sup_{b \in \crochets{0,1}^{V-1}} \esps{b}{\ell\prt{\widehat{g}, g^*_b}}}
\end{equation*}
where $\mathbb{E}_b$ denotes the expectation according to $\P_b$ distribution.

\paragraph*{(ii)} \label{parb2} Let $\widehat{g}$ be a classifier, function of the training sample $\prt{X_i,\Lab_i}_{1\leq i\leq n}$. We use the following decomposition of $\ell$ (\cf\  \eqnref{eql}):
\begin{equation*}
    \ell(\widehat{g},g^*_b) = \esp{\abs{2\eta(X)-1}\abs{\widehat{g}(X)-g^*_b(X)}}\ .
\end{equation*}
Combined with Massart noise condition \cref{$A_2$}, this yields:
\begin{equation*}
    \Rsh \geq h\, \inf_{\widehat{g}\in \S}\brkt{\sup_{b \in \crochets{0,1}^{V-1}} \esps{b}{\abs{\widehat{g}(X)-s_b^*(X)}}}
\end{equation*}

For every $\widehat{g}$, we define $\widehat{b}$ such that:
\begin{equation*}
\widehat{b} = \argmin_{b \in \crochets{0,1}^{V-1}} \esps{X}{\abs{g^*_b(X)-\widehat{g}(X)}}
\end{equation*}
where the expectation is taken with respect to the marginal distribution of $X$ and conditional to the training sample. Hence, $\widehat{b}$ is a function of the training sample $\prt{X_i,\Lab_i}_{1\leq i\leq n}$.
By triangular inequality and then by definition of $\widehat{b}$:
\begin{equation*}
    \abs{g^*_{\widehat{b}}(X)-g^*_b(X)} \leq \abs{g^*_{\widehat{b}}(X)-\widehat{g}(X)} + \abs{\widehat{g}(X)-g^*_b(X)} \leq 2\,\abs{\widehat{g}(X)-g^*_b(X)}\ .
\end{equation*}
Hence:
\begin{equation*}
    \begin{split}
        \Rsh &\geq \frac{h}{2}\, \inf_{\widehat{g}\in \S}\brkt{\sup_{b \in \crochets{0,1}^{V-1}} \esps{b}{\abs{g^*_{\widehat{b}}(X)-s_b^*(X)}}} \\
        &= \frac{h}{2}\, \inf_{\widehat{b}\in \crochets{0,1}^{V-1}}\brkt{\sup_{b \in \crochets{0,1}^{V-1}} \esps{b}{\abs{g^*_{\widehat{b}}(X)-s_b^*(X)}}} \\
        &= \frac{p\, h}{2} \inf_{\widehat{b}\in \crochets{0,1}^{V-1}}\brkt{\sup_{b \in \crochets{0,1}^{V-1}} \esps{b}{\sum_{i=1}^{V-1} \ind{b_i \neq \widehat{b}_i}}} \label{minimax-simplified}
    \end{split}
\end{equation*}
where the last line is obtained by developing the expectation according to the marginal distribution of $X$ which is discrete. \\

\paragraph*{(iii)} \label{parb3} With this simplified expression, we apply Assouad lemma \citep[\cf][]{Yu1997} which provides the following general lower bound:
\begin{equation}\label{assouad}
    \inf_{\widehat{b}\in \crochets{0,1}^{V-1}}\brkt{\sup_{b \in \crochets{0,1}^{V-1}} \esps{b}{\sum_{i=1}^{V-1} \ind{b_i \neq \widehat{b}_i}}} \geq \frac{V-1}{2}\prt{1-\sqrt{\gamma\, n}}
\end{equation}
where $\gamma$ is an upper bound on the square Hellinger distance between probability distributions $\P_b$ and $\P_{b}^\prime$ on $(X,\Lab)$ when $b$ and $b^\prime$ only differ on one coordinate. Using technical Lemma \ref{lemma2} in Appendix \ref{technical}, we have the following upper bound on the square Hellinger distance $\mathcal{H}^2\prt{\P_b, \P_{b^\prime}}$:
\begin{equation}\label{hellinger}
    \mathcal{H}^2\prt{\P_b, \P_{b^\prime}} \leq 2\,p\,e_m\,h^2\ .
\end{equation}

Applying Assouad lemma together with \eqnref{hellinger}, we get the following inequality:
\begin{equation*}
    \Rsh \geq \frac{p\, h}{4}\prt{V-1}\prt{1-\sqrt{2\,p\,e_m\,h^2\,n}}\ .
\end{equation*}

In case \cref{$C_1$}, we choose $p=\frac{2}{9\,e_m\,h^2\,n}$ that is lower than $\frac{1}{V-1}$, we obtain the desired lower bound on $\Rsh$:
\begin{equation*}
    \Rsh \geq \frac{V-1}{54\,e_m\,h\,n}\ .
\end{equation*}
Else, in case \cref{$C_2$}, we choose $p = \frac{2}{9\,e_m\,{\hp}^2\,n}$ where we recall that $\hp = \sqrt{\frac{V}{n\, e_m}}$. As $h \leq \hp$:
\begin{equation*}
    \Rsh \geq \Rshp \geq \frac{V-1}{54\,e_m\,\hp\,n} \geq \frac{1}{54\, \sqrt{2}}\sqrt{\frac{V-1}{n\,e_m}}\ .
\end{equation*}

\begin{flushright}$\blacksquare$\end{flushright}

\subsection{Proof of Proposition \ref{propmm}}\label{proofprop}
This proof relies on the same tools as SCAR assumption case. We alter \nameref{parb1} by choosing $x_1,...,x_V$ satisfying assumption \cref{$A_3$} for $\varepsilon > 0$. \nameref{parb2} remains unchanged. In \nameref{parb3}, the upper bound in \eqnref{hellinger} has to be replaced but $2\,p\, h^2\, \prt{e_m + \varepsilon}$. This yields the following lower bounds:
\begin{enumerate}
\item in case \cref{$C_1$}:
\begin{equation*}
    \Rsh \geq \frac{V-1}{54\,\prt{e_m+\varepsilon}\,h\,n}\ ;
\end{equation*}
\item in case \cref{$C_2$}:
\begin{equation*}
    \Rsh \geq \frac{1}{54\, \sqrt{2}}\sqrt{\frac{V-1}{\prt{e_m+\varepsilon}\,h\,n}}\ .
\end{equation*}
\end{enumerate}
It remains to note that these lower bounds are valid for any $\varepsilon>0$ to complete the proof.\begin{flushright}$\blacksquare$\end{flushright}

\section{Universal entropy metric and related properties}\label{uem}
In this section, we recall some definitions and properties concerning the universal entropy metric. These properties are used for the proof of Proposition \ref{prop3} in Appendix \ref{proofth1}.

Let us consider $(X_i, \Lab_i)_{1\leq i \leq n}$ i.i.d. random variables with values in $\R^d\times\crochets{0,1}$ and $\mathcal{F}$ a set of functions on $\R^d\times\crochets{0,1}$.

\begin{definition}[Universal entropy metric, \cf\ \citealt{Massart2006Oct}]~\\
Let $\varepsilon > 0$ and $\mathbb{Q}$ be a probability measure.\\ Define $h\prt{\mathcal{F}, \varepsilon, \mathbb{Q}}$ as the logarithm of the largest number $N$ of functions $f_1,...,f_N$ separated by a distance $\varepsilon$, namely $\esps{\mathbb{Q}}{\prt{f_i(X,\Lab)-f_j(X,\Lab)}^2}>\varepsilon^2$, $\forall i\neq j$.\\
Then the universal entropy metric $H\prt{\mathcal{F}, \varepsilon}$ is defined as:
\begin{equation}\label{entropy-def}
    H\prt{\mathcal{F}, \varepsilon} = \sup_{\mathbb{Q}}h\prt{\mathcal{F}, \varepsilon, \mathbb{Q}} \ .
\end{equation}
\end{definition}

\begin{proposition}\label{decreasing}
For a fixed $\mathcal{F}$, $H\prt{\mathcal{F}, \cdot}$ is a decreasing function.
\end{proposition}

\begin{proposition}\label{contraction}
Let $\psi$ be a function defined on $\R^d\times\crochets{0,1}$ and $\mathcal{F}$ be a family of functions such that:
\begin{equation*}
    \mathcal{F}=\crochets{(x,s)\mapsto \psi(x,s)\, g(x,s) \ ,\ g \in \mathcal{G}}\ 
\end{equation*}
where $\mathcal{G}$ is another family of functions on $\R^d\times\crochets{0,1}$.
Then:
\begin{equation*}
    \forall \varepsilon>0,\ H\prt{\mathcal{F},\varepsilon} \leq H\prt{\mathcal{G},\frac{\varepsilon}{\abs{\abs{\psi}}_{\infty}}} \ .
\end{equation*}
\end{proposition}
\begin{proof}
Let $\Q$ be a probability distribution and $N$ such that $h\prt{\G, \frac{\varepsilon}{\abs{\abs{\psi}}_{\infty}},\Q} < \log(N)$. Then, for any set of functions $g_1,...,g_N$, there is $i\neq j$ such that $\esps{\Q}{\prt{g_i(X,\Lab)-g_j(X,\Lab)}^2}\leq \prt{\frac{\varepsilon}{\abs{\abs{\psi}}_{\infty}}}^2$. This implies that $\esps{\Q}{\prt{\psi(X,\Lab)\, \brkt{g_i(X,\Lab)-g_j(X,\Lab)}}^2}\leq \varepsilon^2$ and then that $h\prt{\F, \varepsilon, \Q}< \log(N)$.\\
Then, we have that $h\prt{\F, \varepsilon, \Q}\leq h\prt{\G, \frac{\varepsilon}{\abs{\abs{\psi}}_{\infty}},\Q}$. Considering the supremum over the probability distributions $\Q$, we obtain the desired result.
\end{proof}

Finally, we recall Haussler bound which provides an upper bound on the universal entropy metric of a set of classifiers in terms of its VC dimension.

\begin{proposition}[Haussler bound, \cf\ \citealt{Bousquet2003Feb}]\label{haussler}~\\
Assuming that $\F$ is a set of indicator functions with finite Vapnik dimension $V$. Then, $\forall \varepsilon > 0$:
\begin{equation*}
    H\prt{\F, \varepsilon} \leq \kappa\, V\prt{1 + \log\prt{\varepsilon^{-1}\vee 1}}
\end{equation*}
where $\kappa\geq 1$ is a absolute constant.
\end{proposition}

\section{Technical lemmas}\label{technical}
\begin{lemma}\label{lemma1}
Let $C_e>1$ and $\sigma>0$. Then:
\begin{equation*}
    \sum_{j=0}^{+\infty} 2^{-j} \sqrt{1 + \log\prt{2^{j+1} \frac{C_e}{\sigma} \vee 1}} \leq 2\,\prt{1+\log(2)}\sqrt{1 + \log\prt{\frac{C_e}{\sigma} \vee 1}}
\end{equation*}
\end{lemma}
\begin{proof}
\begin{subequations}
\begin{align*}
    \sum_{j=0}^{+\infty} 2^{-j} \sqrt{1 + \log\prt{2^{j+1} \frac{C_e}{\sigma} \vee 1}} & \leq \sum_{j=0}^{+\infty} 2^{-j} \sqrt{1 + (j+1)\log(2) + \log\prt{\frac{C_e}{\sigma} \vee 1}}  \\
    &\leq \sum_{j=0}^{+\infty} 2^{-j} \sqrt{1+(j+1) \log(2)}\sqrt{1+\log\prt{\frac{C_e}{\sigma} \vee 1}}  \\
    &\leq \sum_{j=0}^{+\infty} 2^{-j} \prt{1+(j+1) \frac{\log(2)}{2}}\sqrt{1+\log\prt{\frac{C_e}{\sigma} \vee 1}} \\
    &= 2\,\prt{1+\log(2)} \sqrt{1 + \log\prt{\frac{C_e}{\sigma} \vee 1}} 
\end{align*}
\end{subequations}
\end{proof}

\begin{lemma}\label{lemma2}
Let $x_1,...,x_{V}$ be vectors of $\R^d$. Let $e$ be a function on $R^d$ with values in $(0,1]$. Let $p\leq \frac{1}{V-1}$ and consider $\prt{\P_b}_{b\in \crochets{0,1}^{V-1}}$ the family of probability distributions on $\crochets{x_1,...,x_V}\times\crochets{0,1}$ defined in \nameref{par1} (\cf\ Appendix \ref{proofth2}).
If $b$ and $b^\prime$ are binary vectors of $\crochets{0,1}^{V-1}$ which only differ at coordinate $i$, then:
\begin{equation*}
\mathcal{H}\prt{\P_b,\P_{b^\prime}} \leq 2\,p\,e(x_i)\, h^2\ .
\end{equation*}
\end{lemma}

\begin{proof}
\begin{subequations}
Recall that $b$ and $b^\prime$ only differ at coordinate $i$, hence:
\allowdisplaybreaks
\begin{align}
    \mathcal{H}^2\prt{\P_b, \P_{b^\prime}} &= \frac{1}{2}\sum_{j=1}^{V} \prt{\sqrt{\P_b\prt{X=x_j,\ \Lab=1}} - \sqrt{\P_{b^\prime}\prt{X=x_j,\ \Lab=1}}}^2 \nonumber\\ 
&+ \frac{1}{2}\sum_{j=1}^{V}\prt{\sqrt{\P_b\prt{X=x_j,\ \Lab=0}} - \sqrt{\P_{b^\prime}\prt{X=x_j,\ \Lab=0}}}^2 \nonumber\\
&= \frac{1}{2} \prt{\sqrt{\P_b\prt{X=x_i,\ \Lab=1}} - \sqrt{\P_{b^\prime}\prt{X=x_i,\ \Lab=1}}}^2 \label{eq-h-p1}\\
&+ \frac{1}{2} \prt{\sqrt{\P_b\prt{X=x_i,\ \Lab=0}} - \sqrt{\P_{b^\prime}\prt{X=x_i,\ \Lab=0}}}^2\ . \label{eq-h-p2}
%&= \frac{p}{2}\brkt{e(x_i)\prt{1-\sqrt{1-h^2}} + 2 - e(x_i) - 2 \sqrt{\prt{1-\frac{e(x_i)}{2}(1+h)}\prt{1-\frac{e(x_i)}{2}(1-h)}}}\\
\end{align}
\end{subequations}

Let us now calculate the probabilities using the definition of $\P_b$:
\begin{subequations}
\begin{align*}
    & \P_b\prt{X=x_i, S=1} = p\, \frac{e(x_i)}{2}\brkt{1+\prt{2 b_i - 1}\, h}\ ,\\
    % & \P_{b^\prime}\prt{X=x_i, S=1} = p\, \frac{e(x_i)}{2}\brkt{1+\prt{2 b^\prime_i - 1}\, h}\ ,\\
    & \P_b\prt{X=x_i, S=0} = p\, \prt{1-\frac{e(x_i)}{2}\brkt{1+\prt{2 b_i - 1}\, h}}\ ,\\
%    & \P_{b^{\prime}}\prt{X=x_i, S=0} = p\, \prt{1-\frac{e(x_i)}{2}\brkt{1+\prt{2 b^\prime_i - 1}\, h}}\ .\\
\end{align*}
\end{subequations}

Noting that either $(b_i,b^\prime_i) = (0,1)$ or $(b_i,b^\prime_i) = (1,0)$, we have in both cases:
\begin{equation*}
    \prt{\sqrt{\P_b\prt{X=x_i,\ \Lab=1}} - \sqrt{\P_{b^\prime}\prt{X=x_i,\ \Lab=1}}}^2 = p\, e(x_i)\, \brkt{1 - \sqrt{1-h^2}}\ ,
\end{equation*}
and \eqnref{eq-h-p2},
\begin{align*}
    &\prt{\sqrt{\P_b\prt{X=x_i,\ \Lab=0}} - \sqrt{\P_{b^\prime}\prt{X=x_i,\ \Lab=0}}}^2\\ &= p\brkt{2-e(x_i)-2\sqrt{1-\frac{e(x_i)}{2}\prt{1+h}}\sqrt{1-\frac{e(x_i)}{2}\prt{1-h}}} \ .
\end{align*}
We then sum the two results together:
\begin{subequations}
\allowdisplaybreaks
\begin{align}
\mathcal{H}^2\prt{\P_b, \P_{b^\prime}} &= \frac{p}{2}\brkt{2 - e(x_i)\sqrt{1-h^2} - 2\sqrt{1 - e(x_i) + \frac{e(x_i)^2}{4}\prt{1-h^2}}} \nonumber \\
&= p \brkt{1 - \frac{e(x_i)}{2}\sqrt{1-h^2} - \sqrt{\prt{1 - \frac{e(x_i)}{2}\sqrt{1-h^2}}^2 - e(x_i)\prt{1-\sqrt{1-h^2}}}} \nonumber \\
&= p \brkt{1 - \frac{e(x_i)}{2}\sqrt{1-h^2}} \brkt{1-\sqrt{1-\frac{e(x_i)\prt{1-\sqrt{1-h^2}}}{\brkt{1 - \frac{e(x_i)}{2}\sqrt{1-h^2}}^2}}} \\
&\leq \frac{p\ e(x_i)\prt{1-\sqrt{1-h^2}}}{1 - \frac{e(x_i)}{2}\sqrt{1-h^2}} \\
&\leq 2\, p\, e(x_i)\, h^2 
\end{align}
\end{subequations}
In the above calculation, we applied the inequality $1-\sqrt{1-h^2}\leq h^2$ for $h^2 \in [0,1]$.
\end{proof}
% Note: in this sample, the section number is hard-coded in. Following
% proper LaTeX conventions, it should properly be coded as a reference:

%In this appendix we prove the following theorem from
%Section~\ref{sec:textree-generalization}:

%In this appendix we prove the following theorem from
%Section~6.2:
%
%\noindent
%{\bf Theorem} {\it Let $u,v,w$ be discrete variables such that $v, w$ do
%not co-occur with $u$ (i.e., $u\neq0\;\Rightarrow \;v=w=0$ in a given
%dataset $\dataset$). Let $N_{v0},N_{w0}$ be the number of data points for
%which $v=0, w=0$ respectively, and let $I_{uv},I_{uw}$ be the
%respective empirical mutual information values based on the sample
%$\dataset$. Then
%\[
%	N_{v0} \;>\; N_{w0}\;\;\Rightarrow\;\;I_{uv} \;\leq\;I_{uw}
%\]
%with equality only if $u$ is identically 0.} \hfill\BlackBox
%
%\noindent
%{\bf Proof}. We use the notation:
%\[
%P_v(i) \;=\;\frac{N_v^i}{N},\;\;\;i \neq 0;\;\;\;
%P_{v0}\;\equiv\;P_v(0)\; = \;1 - \sum_{i\neq 0}P_v(i).
%\]
%These values represent the (empirical) probabilities of $v$
%taking value $i\neq 0$ and 0 respectively.  Entropies will be denoted
%by $H$. We aim to show that $\fracpartial{I_{uv}}{P_{v0}} < 0$....\\
%
%{\noindent \em Remainder omitted in this sample. See http://www.jmlr.org/papers/ for full paper.}
%
%
%\vskip 0.2in
\bibliography{biblio}

\end{document}